\documentclass[11pt]{article} 

\usepackage{geometry} 
\usepackage{amsmath,amssymb,amsthm,calrsfs}
\usepackage{enumerate}
\usepackage{fullpage}
\usepackage[nottoc,notlof,notlot]{tocbibind}
\usepackage{aliascnt}
\usepackage{extarrows}
\usepackage[symbol]{footmisc}
\usepackage{subfig}
\usepackage{framed}
\usepackage{multirow}
\usepackage{tikz}
\usetikzlibrary{calc}
\usepackage{graphicx}
\usepackage{xcolor}
\usepackage[mathscr]{euscript}
\usepackage{bbm}
\usepackage[normalem]{ulem}
\usepackage{import}
\usepackage{xifthen}
\usepackage{transparent}

\usepackage{amsrefs}
\usepackage[sans]{dsfont}

\usepackage[colorlinks,linkcolor=blue,citecolor=red,  bookmarksopen=false]{hyperref}

\def\NewTheorem#1#2{%
	\newaliascnt{#1}{thmm}
	\newtheorem{#1}[#1]{#2}
	\aliascntresetthe{#1}
	\expandafter\def\csname #1autorefname\endcsname{#2}
}

\numberwithin{equation}{section}

\newcommand{\boundary}{{\partial{\mathbb{T}}}}

\NewTheorem{lem}{Lemma}
\NewTheorem{thm}{Theorem}
\NewTheorem{cor}{Corollary}
\NewTheorem{prop}{Proposition}
\theoremstyle{definition}
\NewTheorem{defi}{Definition}
\theoremstyle{remark}
\NewTheorem{rem}{Remark}

\newcommand{\Pione}{{\Pi^{(1)}_A}}
\newcommand{\Pioneu}{{\Pi^{(1)}_A(\mathbb{T}^{(u)})}}

\newcommand{\IND}[1] {{ \mathds{1}_{ #1 }} }

\makeatletter
\def\namedlabel#1#2{\begingroup
   \def\@currentlabel{#2}%
   \label{#1}\endgroup
}
\makeatother
\providecommand{\keywords}[1]{\textbf{Keywords.} #1}
\providecommand{\amssc}[1]{\textbf{AMS subject classification.} #1}

\begin{document}
\title{Doubly stochastic Yule cascades (Part II): The explosion problem in the non-reversible case}
\author{Radu Dascaliuc\thanks{Department of Mathematics,  Oregon State University, Corvallis OR, 97331. {\emph{dascalir@math.oregonstate.edu}}}
\and Tuan N.\ Pham\thanks{Department of Mathematics,  Brigham Young University, Provo UT, 84602. {\emph{tuan.pham@mathematics.byu.edu}}}
\and Enrique Thomann\thanks{Department of Mathematics, Oregon State University, Corvallis OR, 97331. {\emph{thomann@math.oregonstate.edu}}}
\and
Edward C.\ Waymire\thanks{Department of Mathematics, Oregon State University, Corvallis OR, 97331. {\emph{waymire@math.oregonstate.edu}}}
}

\maketitle
\begin{abstract} 
We analyze the explosion problem for a class of stochastic models introduced in \cite{part1_2021}, referred to as doubly stochastic Yule cascades.  These models arise naturally in the construction of solutions to evolutionary PDEs as well as in purely probabilistic first passage percolation phenomena having a Markov-type statistical dependence, new for  this context. 
Using cut-set arguments and a greedy
algorithm, we respectively establish criteria for non-explosion and explosion without requiring the time-reversibility of the underlying branching Markov chain (a condition required in \cite{part1_2021}). Notable applications include the explosion of the self-similar cascade of the Navier-Stokes equations in dimension $d=3$ and non-explosion in dimensions $d\ge 12$.

\end{abstract}
.\\\\
\amssc{60H30, 60J80}.

\section{Introduction}\label{motivation}
Largely motivated by the probabilistic method for the deterministic three-dimensional incompressible Navier-Stokes equations (NSE) introduced by Le Jan and Sznitman \cite{lejan}, relaxed in \cite{chaos}, the
 authors introduced a new class of stochastic branching models, referred to as doubly stochastic 
 Yule cascades, in \cite{part1_2021}. This class of branching models may also be viewed in the context of statistically
 dependent first passage percolation models on tree graphs, or in the context
 of certain cellular aging models in biology \cite{DA_PS_1988,KB_PP_2010}.  Roughly speaking, a doubly stochastic Yule cascade evolves from a single progenitor in which each particle waits
for a random length of time depending on the particle's position on the genealogical tree before giving birth to a new
 generation of  particles, each evolving by the same rules as its parent. If the lengths of times between particle reproductions are sufficiently short then branching may produce infinitely many particles in a finite time,
an event referred to as stochastic explosion.
In this paper, we establish some sufficient conditions for explosion and non-explosion, and apply them for the doubly stochastic Yule cascades associated with several well-known differential equations and probabilistic models.

We begin by recalling, in the context of binary trees, some of the definitions introduced in \cite{part1_2021} which contains more general definitions.  With standard notations, debote by
 $\mathbb{T}=\{\theta\}\cup\left(\cup_{n=1}^\infty\{1,2\}^n\right)$ the indexed binary tree with the root $\theta$.   
\begin{defi}[Yule cascade]
A tree-indexed family of independent random variables $\{Y_v\}_{v\in\mathbb{T}}$ is said to be a \emph{Yule cascade} with positive parameters (intensities) $\{\lambda_v\}_{v\in\mathbb{T}}$ if $Y_v\sim \text{Exp}(\lambda_v)$  for every ${v\in\mathbb{T}}$.
\end{defi}
In the above, we used the standard notation $Y_v\sim {\rm Exp}(\lambda_v)$ to denote that $Y_v$ is exponentially distributed with intensity $\lambda_v$. If  $\lambda_v\equiv \lambda>0$ is a constant, this
is the familiar {\it Yule cascade with intensity $\lambda$}. 
The special case $\lambda=1$ is
referred to as the {\it standard Yule cascade}. 
\begin{defi}[Doubly stochastic Yule cascade]
\label{dsydef} 
Consider a tree-indexed family of random variables $Y=\{Y_v\}_{v\in\mathbb{T}}$ and a tree-indexed family of positive random variables $\Lambda=\{\lambda_v\}_{v\in\mathbb{T}}$. $Y$ is said to be a \emph{doubly stochastic Yule \textup{(}DSY\textup{)} cascade} with random intensities $\Lambda$ if $Y$ is, conditionally given $\Lambda$, a Yule cascade with intensities $\Lambda$. 
\end{defi}
Equivalently, $\{Y_v\}_{v\in\mathbb{T}}$ is a DSY cascade with random intensities $\{\lambda_v\}_{v\in\mathbb{T}}$ if and only if $\{T_v=\lambda_vY_v\}_{v\in\mathbb{T}}$ is a family of i.i.d.\ mean one exponentially distributed random variables (called holding times or clocks) independent of $\{\lambda_v\}_{v\in\mathbb{T}}$. Thus, a DSY cascade is completely specified once the family of random intensities is specified. For this reason, one can write a DSY cascade as $\{\lambda_v^{-1}T_v\}_{v\in\mathbb{T}}$. 
The associated counting process defined by
\[N(t)=\left\{ \begin{matrix}
   1 & \text{if} & t\le \frac{{{T}_{\theta}}}{{{\lambda }_{\theta}}},  \\
   \text{card}\left\{ v\in \mathbb{T}:\ \sum\limits_{j=0}^{|v|-1}{\frac{{{T}_{v|j}}}{{{\lambda }_{v|j}}}}<t\le \sum\limits_{j=0}^{|v|}{\frac{{{T}_{v|j}}}{{{\lambda }_{v|j}}}} \right\} & \text{if} & t>\frac{{{T}_{\theta}}}{{{\lambda }_{\theta}}}  \\
\end{matrix} \right.\]
is referred to as the \emph{Yule process} or \emph{doubly stochastic Yule process} depending on the respective
cascade model being considered.  
In the definition of $N(t)$, we have used standard notations for the length of a vertex $v=(v_1,\dots,v_n)\in\mathbb{T}$ and the truncation of $v$ of length $j$, namely $|v|=n$ and  $v|j=
(v_1, \dots, v_j)$ for $0\leq j\leq|v|$.  By convention, $v|0=\theta.$

\begin{defi}[Explosion time]
\label{explosiontime}
The explosion time of a DSY cascade $\{\lambda_v^{-1}T_v\}_{v\in\mathbb{T}}$ is a $[0,\infty]$-valued random variable $\zeta$ defined by
\[\zeta =\underset{n\to \infty }{\mathop{\lim }}\,\underset{|v|=n}{\mathop{\min }}\,\sum\limits_{j=0}^{n}{\frac{{{T}_{v|j}}}{\lambda_{v|j}}}.\]
The event of non-explosion is defined by $[\zeta=\infty]$. The cascade is said to be non-explosive if $\mathbb{P}(\zeta=\infty)=1$, and explosive if $\mathbb{P}(\zeta=\infty)<1$.
\end{defi}

According to the definition, the explosion event $[\zeta<\infty]$ has a positive probability in an explosive cascade. However, in all applications considered in this paper, explosion is in fact a $0-1$ event (\autoref{01}). The DSY cascades introduced in \autoref{dsydef} are quite general. To consider the
explosion problem, we will further assume a certain branching Markov chain structure underlying the random
intensities $\lambda_v$. 

Using estimates on the spectral radius of an operator defined on the underlying branching Markov chains, the authors obtained in  \cite{part1_2021} criteria for non-explosion 
assuming a time-reversibility condition. 
These conditions were applied to the Bessel cascade, a particular cascade associated with the Navier-Stokes equations, and to the cascade of the Kolmogorov-Petrovski-Piskunov (KPP) equation in Fourier domain to show
that they are both non-explosive.  Other more purely probabilistic
models were also analyzed there.  However, time reversibility is too restrictive and does not apply to several examples of particular importance. These include the cascades induced by the scaling properties of the NSE, referred here as self-similar cascades associated with the NSE.  The focus of this paper is to establish explosion and non-explosion 
criteria for a class of DSY cascades using probabilistic arguments that do not depend on the time-reversibility. 

To establish a new non-explosion criterion (\autoref{main}), we rely on the sub-criticality of nonhomogeneous Galton-Watson processes associated with the cascade, which is more intuitive and robust than the method of large deviation estimates used in \cite{part1_2021}. The key idea to obtain non-explosion
is to construct a sequence of {\it recurring finite cut-sets} through an inspection process.  The construction of these cut-sets takes into consideration the uncountable set of paths $s\in\partial{\mathbb{T}}=\{1,2\}^\infty$  in a binary tree, which involves a new notion of recurrence that naturally generalizes neighborhood recurrence along each path (\autoref{cutest-red-thm}). As an application, \autoref{main} implies the non-explosion of the self-similar cascade associated with NSE in dimensions $d\geq 12$ and recovers the non-explosion result for the Bessel cascade of the 3-dimensional NSE and the cascade associated with the KPP equation proven in \cite{part1_2021}.

\begin{rem}
In spite of its intuitive appeal, it is the {\it Markov dependence} along tree paths, in place of {\it independence}, that makes the use of cut-sets technically challenging here.  Similar challenges arise in computing the
speeds of branching random walk particles with Markov-chain 
displacement paths,
in place of independent displacements, and in computing the first passage
percolation \lq\lq flow\rq\rq problems \cites{HK_1986, GG_HK1984}  
if the i.i.d.\ passage times along paths are statistically dependent.   
\end{rem}

We establish a criterion for a.s.\ explosion (\autoref{expcriterion}) based on
a ``fastest path'' constructed by a greedy algorithm: at each time of branching, the fastest path chooses to follow the descendant vertex with the larger intensity $\lambda_v$.  This criterion allows one to show the explosion of a cascade even if the explosion does not occur along any deterministic path $s\in\boundary$, i.e.\
\[\sum\limits_{j=0}^{\infty}{\frac{{{T}_{s|j}}}{\lambda_{s|j}}} = \infty\ \ \ {\text{a.s.}}\ \forall s\in\partial\mathbb{T}.\]
As an application, we obtain the a.s.\ explosion of the self-similar cascade of the 3-dimensional Navier-Stokes equations.  
As shown in \cites{athreya,alphariccati,smallness}, the stochastic explosion of DSY cascades associated with certain evolution equations leads to both nonuniqueness and finite-time blowup of solutions to the initial-value problems in appropriate settings. In particular, the aforementioned explosion of the self-similar DSY cascade corresponding to the Navier-Stokes equations leads to a non-uniqueness result for an equation introduced  by Montgomery-Smith \cite{smith} as a model for possible Navier-Stokes finite-time blowup. Remarkably, our DSY framework also provides an alternative way to prove finite-time blowup of the smooth solutions to the Montgomery-Smith equations \cite{smallness}. For the actual Navier-Stokes equations, a possible resolution of the outstanding nonuniqueness and finite-time blowup problems hinges on a better understanding of the connections between geometric structure of the nonlinear term and the branching structure of the associated DSY cascade.

\section{Main results and organization of the paper}\label{setup}
Motivated by the PDE and probabilistic models mentioned above, we are particularly interested in DSY cascades in which the random intensities $\lambda_v$ are of the form $\lambda_v=\lambda(X_v)$ where, for each $v\in\mathbb{T}$, $X_v$ is a random variable taking values in a measurable space $(S,\mathscr{S})$ and $\lambda:S\to(0,\infty)$ is a measurable function.   Throughout the paper, we assume the following two properties:
\newcounter{saveenum}
\begin{enumerate}[(A)]
\item For any path $s\in\partial\mathbb{T}$, the sequence $X_\theta$, $X_{s|1}$, $X_{s|2}$,\ldots is a time homogeneous Markov chain. 
\item For any path $s\in\partial\mathbb{T}$, the stationary transition probability $p(x,dy)$ does not depend on $s$.
\setcounter{saveenum}{\value{enumi}}
\end{enumerate}
In \cite{part1_2021}, DSY cascades satisfying conditions (A) and (B) are called DSY cascades of type ($\mathscr{M}$). We will refer to the family $\{X_v\}_{v\in\mathbb{T}}$ as the underlying branching Markov chain of the DSY cascade. 

In the statements of the main results, the following conditions are sometimes required.  We use the standard notation $u*v$ for concatenation of vertices $u, v \in \mathbb{T}$.
\begin{enumerate}[(A)]\setcounter{enumi}{\value{saveenum}}
\item For each $v\in\mathbb{T}$, the two subfamilies $\{X_{v* 1* w}\}_{w\in\mathbb{T}}$ and $\{X_{v* 2* w}\}_{w\in\mathbb{T}}$ are conditionally independent of each other given $X_{v* 1}$ and $X_{v*2}$.
\item For each $v\in\mathbb{T}$, the conditional joint distribution $(X_{v*1},X_{v*2})$ given $X_v$ does not depend on $v$. 
\setcounter{saveenum}{\value{enumi}}
\end{enumerate}


The next condition is a generalization of the neighborhood recurrence along a path to Markov branching process.  First, for $a\in S, A\in\mathscr{S},$ $n\geq 1$ and $s\in\partial\mathbb{T}$,
define
\begin{equation}\label{I_n}
I_n(a,A)=\mathbb{P}_a(X_{s|1}\not\in A,\,X_{s|2}\not\in A,\ldots,\,X_{s|n}\not\in A).
\end{equation}
Notice that $I_n$ does not depend on $s\in\partial\mathbb{T}$ by virtue of Properites (A) and (B). 
\begin{enumerate}[(A)]\setcounter{enumi}{\value{saveenum}}
\item { (Cut-set recurrence condition) There exists a set $A\in \mathscr{S}$, a function $\psi:S\to\mathbb{R}$ and a number $r>2,$ such that for all $a\in S$, $s\in\partial\mathbb{T}$ and
$n\in\mathbb{N}$,
$$I_n(a,A)\le \psi (a)r^{-n}.$$  
}
\end{enumerate}
While conditions (C) and (D) are intuitive, condition (E) is somewhat technical. It implies a kind of recurrence property for the branching Markov chain; in the language to be made clear later, the set $A$ is \emph{cut-set recurrent} with respect to $\{X_v\}_{v\in\mathbb{T}}$. The intuition is that the Markov chain $\{X_{s|j}\}_{j\ge 0}$ returns to the set $A$ infinitely often in a uniform manner across all the paths $s\in\partial\mathbb{T}$ (cf.\ \autoref{cutset-def-sec} for the precise definition). 
Specifically, we have the following result.

\begin{thm}\label{cutest-red-thm}
Assume that the tree-indexed family of random variables 
$\{X_v\}_{v\in\mathbb{T}}$ satisfies \textup{(A), (B), (C), (E)}. Then, with $A$ and $\psi$ given by condition \textup{(E)}, we have:
\begin{enumerate}[(i)]
\item  The set $A$ is cut-set recurrent with respect to $\{X_v\}_{v\in\mathbb{T}}$ in the sense of \autoref{cutset-rec-def}. 

\item  The cardinality of the first passage cut-set is a.s. finite and satisfies 
\[\mathbb{E}_a[{\rm card}\,\Pi_A^{(1)}]\le2+\psi(a)\frac{2r}{r-2}\,.\]
Moreover, if  $\psi$ is bounded on $A$ 
then the expected valued of the cardinality of each passage cut-set is bounded by
\[\mathbb{E}_a[{\rm card}\, \Pi_A^{(k)}]\leq \mu(a)^k\qquad \forall\, k\in\mathbb{N}.\]
where $\mu(a) = \sup\{\psi(x):\,x\in A\cup\{a\}\}.$
\end{enumerate}
\end{thm}
The proof will be given in \autoref{proof}.
This theorem leads to a strategy to show the non-explosion of a DSY cascade by finding an appropriate value $c>0$ such that the set $A=\lambda^{-1}((0,c])$ is 
cut-set recurrent.  
Next, we state the main criterion for non-explosion.  We say that a function $\phi: (0,\infty) \rightarrow (0,\infty)$ is \emph{locally bounded} if it maps every bounded set into a bounded set.

\begin{thm}\label{main}
Let  $\{X_v\}_{v\in\mathbb{T}}$ be a branching Markov chain with values on a measurable state space $(S,\mathscr{S})$ and satisfy \textup{(A), (B), (C)}. Let $\{\lambda(X_v)^{-1}T_v\}_{v\in\mathbb{T}}$ be a DSY cascade with $\lambda:S\to(0,\infty)$ a measurable function. Suppose that condition \textup{(E)} holds for $A=\lambda^{-1}((0,c])$ for some $c>0$ and $\psi=\phi\circ\lambda$, where $\phi:(0,\infty)\to(0,\infty)$ is locally bounded. Then the DSY cascade is non-explosive for any initial state $X_\theta=a\in S$. 
\end{thm}


The proof of this theorem is in \autoref{Theorem2}.
In our applications, the case $S=(0,\infty)$ is of particular interest. In this case, we have the following corollary.

\begin{cor}\label{main_cor}
Let $\{\lambda(X_v)^{-1}T_v\}_{v\in\mathbb{T}}$ be a DSY cascade in which $\{X_v\}_{v\in\mathbb{T}}$ has properties \textup{(A), (B), (C)} on a measurable state space $S=(0,\infty)$ (and $\mathscr{S}$ is the Borel $\sigma$-algebra). Suppose that the condition \textup{(E)} holds for $A=(0,c]$ for some $c>0$, with $\lambda$ and $\psi$ both being locally bounded.  Then the DSY cascade $\{\lambda(X_v)^{-1}T_v\}_{v\in\mathbb{T}}$ with $X_\theta=a$ is non-explosive for any $a\in S.$
\end{cor}

\begin{rem}\label{X>0}
Note that in the context of \autoref{main_cor}, $I_n(a,A)$ from \eqref{I_n} becomes
\begin{equation}\label{I_n'}
{I}_n(a,c)=\mathbb{P}_a(X_{s|1}>c,\,X_{s|2}>c,\ldots,\,X_{s|n}>c).
\end{equation}
\end{rem}


We will give in \autoref{corproof} two sufficient conditions (relatively simple to check) for condition (E) to be satisfied. 
An interesting class of DSY cascades is the \emph{self-similar DSY cascades}. This class includes the cascade of the $\alpha$-Riccati equation, the cascade of the complex Burgers equation, and the cascade of the Navier-Stokes equations with a scale-invariant kernel (\autoref{NSEExplosion}). 
\begin{defi}
Let $S$ be a subgroup of $(0,\infty)$. A DSY cascade $\{\lambda(X_v)^{-1}T_v\}_{v\in\mathbb{T}}$ is said to be \emph{self-similar} if the branching Markov chain $\{X_v\}_{v\in\mathbb{T}}$ satisfies the following:
\begin{enumerate}[(i)]
\item $\{X_v\}_{v\in\mathbb{T}}$ satisfies (A), (B), (C).
\item For any initial state $X_\theta=a\in S$ and $v\in\mathbb{T}$, the family $\{Y_w=X_{v*w}/X_v\}_{w\in\mathbb{T}}$ also satisfies (A), (B), (C) with the same transition probability as $\{X_v\}_{v\in\mathbb{T}}$.
\end{enumerate}
\end{defi}

In \autoref{selfsimilarsection}, we provide a useful characterization of self-similar DSY cascades: these are exactly the DSY cascades in which the branching Markov chain $\{X_v\}_{v\in\mathbb{T}}$ satisfies (A), (B), (C) such that along each path $s\in\partial\mathbb{T}$, the sequence $\{X_{s|j}\}_{j\ge 0}$ is a multiplicative random walk on $S$ (i.e.\ $\ln X_{s|j}$ is an additive random walk on the additive group $\ln S$). Multiplicative random walks on $(0,\infty)$ are natural examples of non-reversible Markov chains, which are of interest in this paper. Self-similar DSY cascades admit a rather simple criterion for non-explosion:
\begin{prop}\label{multibran}
Let $\{\lambda(X_v)^{-1}T_v\}_{v\in\mathbb{T}}$ be a self-similar DSY cascade in which $\lambda:S\to(0,\infty)$ is a locally bounded function. Suppose $\mathbb{E}[R_1^b]<1/2$ for some $b>0$. Then the cascade is non-explosive for any initial state $X_\theta=a\in S$.
\end{prop}
The proof of this result is given in \autoref{selfsimilarsection}.
The main criterion for a.s.\ explosion of DSY cascades is as follows.  
\begin{thm}\label{expcriterion}
Let $\{\lambda(X_v)^{-1}T_v\}_{v\in\mathbb{T}}$ be a DSY cascade in which $\{X_v\}_{v\in\mathbb{T}}$ satisfies \textup{(A), (B), (D)}. Put $Z=\max\{\lambda(X_1),\lambda(X_2)\}$. If there exists a constant $\kappa\in (0,1)$ such that
\[\mathbb{E}_a\left[\frac{1}{Z}\right]\le \frac{\kappa}{\lambda(a)}\ \ \ \forall\, a\in S\]
then ${{\mathbb{E}}_{a}}\zeta \le {\lambda (a)^{-1}}{(1-\kappa)^{-1}}$. In particular, the cascade is a.s.\ explosive for any initial state $X_\theta=a\in S$.
\end{thm}
The proof this theorem is given in \autoref{explosioncriteria}. In \autoref{NSEExplosion}, it is shown that the explosion criterion is satisfied by the self-similar DSY cascades associated with the 3-dimensional Navier-Stokes equations. 
We also show that for large spatial dimensions ($d\geq 12$), the self-similar DSY cascades are non-explosive. The range $3<d<12$ remains inconclusive to us. Nevertheless, we show that the explosion event in a self-similar DSY cascade is a $0-1$ event. More generally, one has the following.  
 \begin{thm}\label{01}
Let $\{\lambda(X_v)^{-1}T_v\}_{v\in\mathbb{T}}$ be a DSY cascade in which $\{X_v\}_{v\in\mathbb{T}}$ has property \textup{(A), (B), (C), (D)}. Then either $\mathbb{P}_x(\zeta=\infty)=0$ for all $x\in S$ or $\mathbb{P}_x(\zeta=\infty)=1$ for all $x\in S$ if one of the following conditions holds.
\begin{enumerate}[(i)]
\item $\{\lambda(X_v)^{-1}T_v\}_{v\in\mathbb{T}}$ is a self-similar DSY cascade and $\lambda$ is a multiplicative function, i.e.\ $\lambda(xy)=\lambda(x)\lambda(y)$ for all $x,y\in S$.
\item Along each path $s\in\partial\mathbb{T}$, $\{X_{s|j}\}_{j\ge0}$ is an ergodic time-reversible Markov chain on a countable state space $S$.
\item Along each path $s\in\partial\mathbb{T}$, $\{X_{s|j}\}_{j\ge0}$ is time-reversible with respect to a probability measure $\gamma(dx)$ (i.e.\ $\gamma(dx)p(x,dy)=\gamma(dy)p(y,dx)$ as measures on $S\times S$) which satisfies $\gamma(dy)\ll p(x,dy)\ll\gamma(dy)$ for every $x\in S$.
\end{enumerate}
\end{thm}
The proof of this result is given in \autoref{zeroone}. Applications of the non-explosion and explosion criteria are given in the last three sections of the paper.  In  \autoref{examples}, we focus on examples originating from probability, whereas \autoref{diff_eq_examples} gives examples originating from differential equations.  The applications to NSE are of particular interest and are mentioned in a separated section (\autoref{NSEExplosion}).
\section{Preliminary notions}\label{cutset-def-sec}
To prepare for the proofs of the main results, we introduce some preliminary notions on a general rooted tree (see also \cite{lyons90, lyons92}, \cite[Sec.\ 2.5]{lyons}).

Let $\mathcal{T}\subset\mathbb{V} = \{\theta\}\cup(\cup_{n\ge 1}\mathbb{N}^n)$ be a connected locally finite random tree rooted at $\theta$. 
Each vertex $v=(v_1,\ldots,v_n)\in\mathcal{T}$ is connected to the root by a unique path, so one can identify a vertex with the path connecting it to the root.  For a finite or infinite path $s=(s_1,s_2,\ldots)$, we denote by $v* s$ the path obtained by appending $s$ to $v$:
\[v* s=(v_1,\ldots,v_n,s_1,s_2,\ldots)\]
The length of a path $v=(v_1,\ldots,v_n)$, or the genealogical height of vertex $v$, is $|v|=n$. For $1\le j\le n$, the truncation of $v$ up to the $j$'th generation is $v|j=(v_1,\ldots,v_j)$. We use the convention that $v|0=\theta$ and $|\theta|=0$. The boundary of $\mathcal{T}$ is defined as the set of infinite paths:
\[\partial\mathcal{T}=\{s\in\mathbb{N}^\infty:~s|j\in\mathcal{T},\ \forall\,j\ge 0\}.\]
\begin{rem}
One can interpret $\partial\mathcal{T}$ as the boundary of $\mathcal{T}$ in the metric space $\mathcal{T}\cup\mathbb{N}^\infty$ endowed with the metric $d(v,w)=\sum^{\infty}_{j=1}2^{-j}(1-\delta_{\alpha_j,\beta_j})$ where $\delta_{x,y}$ is the standard Kronecker delta and
\[{{\alpha }_{j}}=\left\{ \begin{array}{*{35}{r}}
   {{v}_{j}} & \text{if} & j\le |v|  \\
   0 & \text{if} & j>|v|  \\
\end{array} \right.,\ \ {{\beta }_{j}}=\left\{ \begin{array}{*{35}{r}}
   {{w}_{j}} & \text{if} & j\le |w|  \\
   0 & \text{if} & j>|w|  \\
\end{array} \right..\]
\end{rem}
\begin{defi}[cut-sets]
A finite set of vertices $V\subset \mathcal{T}\backslash\{\theta\}$ is called a \emph{cut-set} of $\mathcal{T}$ if for each path $s\in\partial \mathcal{T}$, there exists unique $j\ge 1$ such that $s|j\in V$.
\end{defi}
Intuitively, a cut-set is a set of vertices that separates the root from the boundary.
\begin{defi}[Passage sets]
Let $\{Y_v\}_{v\in\mathbb{V}}$ be a tree-indexed family of random variables independent of $\mathcal{T}$ taking values on a measurable state space $({S},\mathscr{S})$. For $A\in\mathscr{S}$, we define a sequence of random sets $\{\Pi_A^{(n)}\}_{n\ge 1}$ depending on $\mathcal{T}$ as follows.
\[\Pi_A^{(1)}  =\{v\in\mathcal{T}\backslash\{\theta\}:\ Y_v\in A,\ Y_{v|j}\not\in A~~\forall\, 0<j<|v|\},\]
\[\Pi_A^{(n)}\ =\{v\in\mathcal{T}\backslash\{\theta\}:\ Y_v\in A,\ \exists 1\le j< |v|,\ v|j\in\Pi_A^{(n-1)},\ Y_{v|i}\not\in A\ \forall\, j<i<|v|\}\]
for all $n\ge 2$. We call $\Pi_A^{(n)}$ the $n$'th passage set of $\{Y_v\}_{v\in \mathcal{T}}$ through $A$.
\end{defi}

\begin{defi}[Cut-set recurrence]\label{cutset-rec-def}
Let $\{Y_v\}_{v\in\mathbb{V}}$ be a tree-indexed family of random variables on a measurable state space $({S},\mathscr{S})$. A set $A\in\mathscr{S}$ is said to be \emph{cut-set recurrent} with respect to $(\{Y_v\}_{v\in\mathbb{V}}, \mathcal{T})$ if each passage set 
$\Pi_A^{(1)}$, $\Pi_A^{(2)}$, $\Pi_A^{(3)}$, \ldots is almost surely a cut-set. In this case, $\Pi_A^{(n)}$ is called the $n$'th \emph{passage cut-set}.
\end{defi}
For the sake of simplicity, we will write the structure $(\{Y_v\}_{v\in\mathbb{V}}, \mathcal{T})$ simply as $\{Y_v\}_{v\in\mathcal{T}}$.
\section{Proof of the \autoref{cutest-red-thm}}\label{proof}
With the terminology introduced in the previous section, the proof of \autoref{cutest-red-thm} uses the decay of $I_n\lesssim r^{-n}$ and an inspection process on whether $X_v\notin A$ to construct a sequence of passage cut-sets. 
Specifically 
at each vertex $v\in\mathbb{T}$, starting from the root $\theta$, we inspect each offspring $u\in\{v*1,\,v*2\}$ of $v$. If $X_u\notin A$  then $u$ passes the inspection. Otherwise, if $X_u\in A$, $u$ does not pass the inspection and the inspection process along the path containing $u$ stops at $u$. 
For each offspring of $v$ that passes the inspection, we proceed to inspect its own offspring. 
For example, if  $X_{v* 1}\notin A$ and $X_{v* 2} \notin A$ then we inspect the vertices $v* 11$, $v* 12$, $v* 21$, $v* 22$. If $X_{v* 1}\notin A$ and $X_{v* 2}\in A$, we only inspect $v* 11$ and $v* 12$. In this manner, the inspection process either continues indefinitely or stops after finitely many steps. Note that we do not inspect the root: $\theta$ is considered passing the inspection whether or not $X_\theta\notin A$.

For a vertex $v\in \mathbb{T}$ that passes the inspection, we denote by $\mathcal{O}_v$ the number of offspring of $v$ that passes the inspection. The distribution of $\mathcal{O}_{v=\theta}$ is given by
\begin{eqnarray*}
\mathbb{P}_a\left(\mathcal{O}_\theta=2\right)&=&\mathbb{P}_a\left(X_1\notin A,\,X_2\notin A\right),\\
\mathbb{P}_a(\mathcal{O}_\theta=1)&=&\mathbb{P}_a(X_1\notin A,\,X_2 \in A)+\mathbb{P}_a(X_1\in A,\,X_2\notin A),\\
\mathbb{P}_a(\mathcal{O}_\theta=0)&=&\mathbb{P}_a(X_1\in A,\,X_2 \in A).
\end{eqnarray*}

For $v\in\mathbb{T}\setminus\{\theta\}$, let  $B_v = [X_{v|1}\notin A, X_{v|2}\notin A,\dots, X_v\notin A]$ 
so that $I_n(a,A)= \mathbb{P}_a(B_v)$. Note that this probability only depends on $n=|v|$ due to properties (A) and (B).
Let 
\begin{equation}
\label{Nofa}
N(a)=\inf\{n\ge 1:~I_n(a,A)=0\}
\end{equation}
which could be infinity. For each vertex $v$ with $1\le |v|=n< N(a)$, the number of offspring passing the inspection has a distribution
\begin{eqnarray}
\label{eq:510211}\mathbb{P}_a(\mathcal{O}_v=2)&=&\mathbb{P}_a(X_{v* 1}\notin A,\,\,X_{v* 2}\notin A\ |\ B_v),\\
\label{eq:429211}
\mathbb{P}_a({{\mathcal{O}}_{v}}=1)&=&\mathbb{P}_a(X_{v* 1}\notin A,X_{v* 2}\in A\ |\ B_v)
+\mathbb{P}_a(X_{v* 1}\in A, X_{v* 2}\notin A\ |\ B_v),\\
\label{eq:510212}\mathbb{P}_a(\mathcal{O}_v=0)&=&\mathbb{P}_a(X_{v* 1}\in A,\,\,X_{v* 2}\in A \ |\ B_v).
\end{eqnarray}

Invoking Properties (A) and (B) once again, 
one can write $\mathcal{O}_v \sim \mathcal{O}_n$ for $n=|v|.$ 
Denote $p_{n,i}(a)=\mathbb{P}_a(\mathcal{O}_n=i)$ for $i=0,1,2$. Let $Z_n$ be the total number of individuals at generation $n$ that pass the inspection. Then $Z_0= 1$ and
\[{{Z}_{n+1}}=\sum\limits_{j=1}^{{{Z}_{n}}}{{\mathcal{O}_{n,j}}}\ \ \ \forall\, 0\le n<N(a)\]
where $\mathcal{O}_{n,1}$, $\mathcal{O}_{n,2}$, $\mathcal{O}_{n,3}$,\ldots are i.i.d copies of $\mathcal{O}_n$. The inspection process can be viewed as a nonhomogeneous Galton-Watson process with the offspring distribution given above. We will show this process stops eventually. 

\begin{lem}\label{stop}
Under the assumption of \autoref{cutest-red-thm},
for every $a>0$, the inspection process starting at $X_\theta=a$ almost surely stops after finitely many steps.
\end{lem}
\begin{proof}
We have
\[\mu_0:=\mathbb{E}\mathcal{O}_0=2\mathbb{P}_a(\mathcal{O}_0=2)+\mathbb{P}_a(\mathcal{O}_0=1)=\mathbb{P}_a(X_1\notin A)+\mathbb{P}_a(X_2\notin A)=2I_1.\]

If $N(a)=1$ then {$\mathbb{P}_a(X_1\notin A)=\mathbb{P}_a(X_2\notin A)=0$}. In this case, the inspection process stops at the root. Consider the case $2\le N(a)\le\infty$. For $1\le n<N(a)$, denote the two terms on the right hand side of \eqref{eq:429211} by $\bar{p}_{1,n}$ and $\tilde{p}_{1,n}$. 
\begin{eqnarray*}{{\mu }_{n}}:=\mathbb{E}{\mathcal{O}_{n}}&=&2{{p}_{n,2}}+{\bar{p}_{n,1}}+\tilde{p}_{n,1}=({{p}_{n,2}}+{\bar{p}_{n,1}})+({{p}_{n,2}}+{\tilde{p}_{n,1}})\\
&=&\mathbb{P}(X_{v* 1}\notin A\,|\,B_v)+\mathbb{P}(X_{v* 2}\notin A\,|\,B_v)\\
&=&2\mathbb{P}(X_{v* 1}\notin A\,|\,B_v)\\
&=&2\frac{{{I}_{n+1}}}{{{I}_{n}}}.\end{eqnarray*}

By Wald's identity, $\mathbb{E}{{Z}_{n}}=\mu_{n-1}\mathbb{E}{{Z}_{n-1}}$.
Applying this identity consecutively, we get
\begin{equation}
\label{eq:925201}
\mathbb{E}{{Z}_{n}}={{\mu }_{n-1}}...\mu_1\mathbb{E}Z_0=\mu_{n-1}\ldots\mu_1\mu_0={{2}^{n}}{I_{n}}\le 
\psi(a)
\left(\frac{2}{r}\right)^n~~~\forall\,1\le n< N(a),\,n<\infty.
\end{equation}
If $N(a)<\infty$ then $\mathbb{P}_a(X_{v|1}\notin A,\,X_{v|2}\notin A,\ldots,X_{v|n}\notin A)=0$ for all $v\in\mathbb{T}$, $|v|=n\ge N(a)$. Hence, the inspection process stops almost surely after $N(a)$ generations. Consider the case $N(a)=\infty$. Put $Z_\infty=\liminf Z_n$. By Fatou's lemma,
\[\mathbb{E}{{Z}_{\infty }}\le \underset{n\to \infty }{\mathop{\liminf }}\,\mathbb{E}{{Z}_{n}}\le\underset{n\to \infty }{\mathop{\liminf }}\,
\psi(a)
\left(\frac{2}{r}\right)^n=0.\]
Hence, $Z_\infty=0$ a.s. Because $Z_n$'s are nonnegative integers, $Z_n=0$ a.s. for some random value of $n$. 
\end{proof} 

To complete the proof of Part (i) of \autoref{cutest-red-thm},
we use \autoref{stop} to show that the 
set $A$
is cut-set recurrent with respect to the branching Markov chain $\{X_v\}_{v\in\mathbb{T}}$.
By definition, the first passage set $\Pi_A^{(1)}$ consists, for each $s\in\partial\mathbb{T}$, of the first vertex $v\neq\theta$ such that 
{\color{blue}$X_{v}\in A$} 
(see \autoref{cutsets}). In other words, it consists of the vertices at which the inspection process stops. By \autoref{stop}, $\Pi_A^{(1)}$ is, a.s, a finite set and, thus, must be a cut-set. 

Suppose that for some $k\ge 1$, the passage set $\Pi_A^{(k)}$ is a cut-set. 
For each $u\in\mathbb{T}$, denote by $\mathbb{T}^{(u)}$ the subtree of $\mathbb{T}$ rooted at $u$ and write $X^{(u)}_v=X_{u* v}$ for $v\in\mathbb{T}$.
By conditions (A) and (B), we have that condition (E) is also satisfied by the branching Markov chain $X^{(u)}_v$.  Indeed, the analogue to \eqref{I_n} for this branching Markov process is
\begin{eqnarray*}
I_n(a',A;X^{(u)}) &\equiv& 
\mathbb{P}(X_{u*s|(|u|+1)}\notin A, X_{u*s|(|u|+2)}\notin A, \dots, X_{u*s|(|u|+n)}\notin A|X_u=a') \\
&=&\mathbb{P}_{a'}(X_{s|1}\notin A, X_{s|2}\notin A, \dots, X_{s|n}\notin A) \\
&=&
I_n(a',A)
\end{eqnarray*}
Given $X_u$, by \autoref{stop} the inspection process on $\mathbb{T}^{(u)}$ must terminate a.s. Thus, for each path $s\in\partial\mathbb{T}$ passing through $u$, there exists a random integer $j$, $|u|< j<\infty$ such that 
$X_{s|j}\in A$.
Let $N_{u,s}$ be the smallest value of such $j$'s. Let 
\begin{equation}\label{Cu}
\Pioneu=\left\{ s|_{{N}_{u ,s}}:\ s\in {\partial{\mathbb{T}}}~~\text{passing through } u \right\}
\end{equation}
corresponding to the first passage set of the tree rooted at $u$ through $A$.
Thus, by \autoref{stop}, $\Pioneu$ is a.s. finite and nonempty and separates boundary paths in $\partial\mathbb{T}$ that pass through $u$ from $u$.
Note that, by definition, the $k+1$'st passage set is given
by 
\begin{equation}
\label{piku}
\Pi_A^{(k+1)}=\bigcup_{u\in\Pi_A^{(k)}}\Pioneu
\end{equation}
which is nonempty, finite almost surely. Since $\Pi_A^{(k)}$ is a cut-set, $\Pi_A^{(k+1)}$ is also a cut-set.
We conclude that 
$A$ is cut-set recurrent thus completing the proof Part (i) of \autoref{cutest-red-thm}.

\begin{figure}
  \centering
  \includegraphics[scale=1]{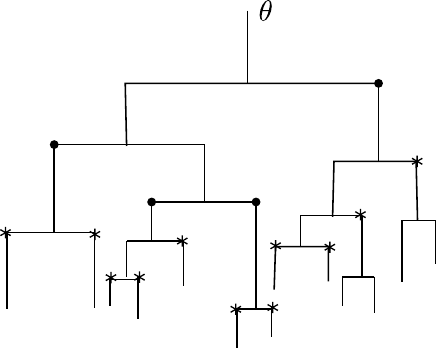}
  \caption{$\Pi_A^{(1)}$ consists of the bold dots. $\Pi_A^{(2)}$ consists of the stars.}
  \label{cutsets}
\end{figure}

%
%
%



%
%
%
%

We begin the proof of Part (ii) of  \autoref{cutest-red-thm} by showing the estimate for the mean cardinality of $\Pi_A^{(1)}$, namely
\begin{equation}\label{eq:510214}\mathbb{E}_a[{\rm card}\,\Pi_A^{(1)}]\le2+\psi(a)\frac{2r}{r-2}\end{equation}
and use induction for the other estimate.
With $N(a)$ defined in \eqref{Nofa}, if $N(a)=1$ then 
$\Pi_A^{(1)}=\{1,2\}$
a.s. and $\eqref{eq:510214}$ holds. Consider the case $N(a)\ge 2$. The mean total number of individuals passing the inspection process is
\[\mathbb{E}_a\left[ \sum\limits_{n=0}^{N(a) }{{{Z}_{n}}} \right]=\sum\limits_{n=0}^{N(a) }{\mathbb{E}{{Z}_{n}}}\le 1+\psi(a)\sum\limits_{n=1}^{N(a) }\left(\frac{2}{r}\right)^n\le 1+\psi(a)\frac{r}{r-2}\,,\]
according to \eqref{eq:925201}. Because each vertex in $\Pi_A^{(1)}$ is an offspring of a vertex that passes the inspection process, card\,$\Pi_A^{(1)}$ is at most twice the total number of vertices passing the inspection process. Therefore,
\[\mathbb{E}_a[{\rm card}\,{\Pi_A^{(1)}}]\le 2\mathbb{E}\left[ \sum\limits_{n=0}^{\infty }{{{Z}_{n}}} \right]\le 2+\psi(a)\frac{2r}{r-2}  .\]
Consider now $u\in\Pi_A^{(k)}$ for some $k\ge 1$, 
so that $X_u=a'\in A.$ 
Then, by properties (A) and (B) of the branching Markov chain we have
\begin{equation}
\nonumber
{{\mathbb{E}}_{a}}[\operatorname{card}(\Pioneu)|X_u=a']={
\mathbb{E}}_{X_u=a'}[\operatorname{card}(\Pioneu)]
=
{{\mathbb{E}}_{{{X}_{\theta }}=a'}}[\operatorname{card}(\Pi^{(1)}_A)]\le 2+\psi(a')\frac{2r}{r-2}~~~
\end{equation}
Recall that we now assume that $\psi$ is bounded on $A$ and that $M(a)=\sup\left\{\psi(x):\,x\in A\cup\{a\}\,\right\}.$
Then
\begin{equation}
\label{eq:923201}
{{\mathbb{E}}_{a}}[\operatorname{card}({\Pioneu})|X_u=a']
\leq
2+M(a)\frac{2r}{r-2} \equiv \mu(a)
\end{equation}
The proof of \autoref{cutest-red-thm} is completed by noting that
\[
\mathbb{E}_a[{\rm card}\,{\Pi_A^{(k)}}] =
\mathbb{E}_a[\mathbb{E}_a[{\rm card}\,{\Pi_A^{(k)}}|\Pi_A^{(k-1)}]] \leq
\mu(a) \mathbb{E}_a[{\rm card}\,{\Pi_A^{(k-1)}}] \leq (\mu(a))^k.
\]

 
\section{Proof of \autoref{main}}\label{Theorem2}

Thanks to \autoref{cutest-red-thm}, we can define a random sequence of nonnegative integers $\{H_k\}$ where
$H_0=1$ and
\begin{equation}
\label{N_k}
H_{k}=\max\{|v|:~v\in\Pi_A^{(k)}\}\ \ \ \forall\,k\in\mathbb{N}.
\end{equation}
One can observe that for any $k\in\mathbb{N}$ and $s\in \partial\mathbb{T}$, there exists $l_k=l_k(s)\in[k,\, H_k)$ satisfying $s|l_k\in\Pi_A^{(k)}$.  In order to establish the non-explosive character of the DSY cascade, note that
for $k\in\mathbb{N}$ and with $n\ge H_k$, we have:
\[{{\zeta}_{n}} = \underset{|v|=n}{\mathop{\min }}\,\sum\limits_{j=1}^{n}\frac{T_{v|j}}{\lambda(X_{v|j})}\ge \underset{|v|=n}{\mathop{\min }}\,\sum\limits_{i=1}^{k}{\frac{{{{{T}}}_{v|{{l}_{i}}}}}{\lambda(X_{v|{{l}_{i}}})}}\ge \frac{1}{c}\underset{|v|=n}{\mathop{\min }}\,\sum\limits_{i=1}^{k}{{{{{T}}}_{v|{{l}_{i}}}}}\]
where 
$l_i=l_i(s_v)$ with $s_v\in\partial\mathbb{T}$ such that $s_v|_{|v|}=v$. In the last inequality we have used that $A=\lambda^{-1}(0,c)$ so that since $X_{v|l_i}\in A$ one has $\lambda(X_{v|l_i})< c.$
Since $\{T_v\}_{v\in\mathbb{T}}$ is independent of  $\{X_v\}_{v\in\mathbb{T}}$, in order to establish that the DSY cascade is not explosive it suffices to show that almost surely,
\begin{equation}\label{eq1:11420}\underset{k\to \infty }{\mathop{\lim }}\,\underset{|v|=H_k}{\mathop{\min }}\,\sum\limits_{i=1}^{k}{{{{{T}}}_{v|{{l}_{i}}}}}=\infty
\end{equation}
For this, we define a random tree, referred as the  
``reduced tree'', consisting of the root $\theta$ and vertices in the passage cut-sets $\Pi_A^{(1)}$, $\Pi_A^{(2)}$, $\Pi_A^{(3)}$, \ldots Throughout this section we will use
\begin{equation}
\label{cardinality}
K_u=\textrm{card}\,\Pioneu
\end{equation}
denote the random number of vertices in the first cut-set of the tree rooted at $u.$
For the reduced tree, the root $\theta$ has $K_\theta=\textrm{card}\,\Pione$ offspring, each of which is a vertex in the cut-set $\Pione$. Each vertex $u$ of the first generation of the reduced tree has $K_u$ offspring, each of which is a vertex in the cut-set $\Pioneu$, and so on (\autoref{reduced-cutset}). Thus, the first generation of the reduced tree consists of vertices in the first cut-set $\Pi_A^{(1)}$. The second generation of this tree consists of vertices in the second cut-set $\Pi_A^{(2)}$. In general, the $k$'th generation of the reduced tree consists of vertices in the $k$'th cut-set $\Pi_A^{(k)}$. 


\begin{figure}
  \centering
  \includegraphics[scale=0.9]{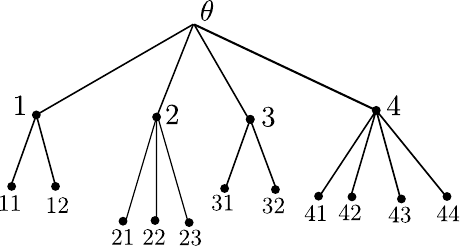}
  \caption{The reduced tree $\mathcal{T}$ corresponding to the binary tree $\mathbb{T}$ in \autoref{cutsets}.}
  \label{reduced-cutset}
\end{figure}

The set of all vertices of the random reduced tree is denoted by $\mathcal{T}\subset\mathbb{V} = \{\theta\}\cup(\cup_{n=1}^\infty \mathbb{N}^n)$. Recall that the configuration of $\mathcal{T}$ is 
independent of the clocks $\{T_v\}_{v\in\mathbb{V}}$ after a natural relabeling of the vertices (\autoref{reduced-cutset}). 
Now \eqref{eq1:11420} becomes a non-explosion problem of a DSY cascade $\{T_v\}_{v\in\mathbb{V}}$ with all intensities equal to 1 on a random tree structure $\mathcal{T}$. 
The uniform bound obtained in Part (ii) of \autoref{cutest-red-thm} on the expected number of offspring of each vertex turns out to be crucial for our approach. For further criteria for non-explosion of a DSY cascade on a random tree structure, see \cite[Sec.\ 4]{part1_2021}.  


To show \eqref{eq1:11420}, 
we will use a cut-set argument similar to the one used in the proof of \autoref{cutest-red-thm}. Let $\varepsilon>0$ be a number to be chosen. Given the random tree $\mathcal{T}$, we start an inspection process of whether $T_v\le \varepsilon$ as follows. At each vertex $v\in\mathcal{T}$, starting from the root $\theta$, we inspect its offspring. If an offspring $u\in\{v* 1,\,v* 2,\ldots,\,v* K_v\}$ satisfies $T_u\le\varepsilon$ then it passes the inspection. Otherwise, it does not pass the inspection and the inspection process along the path containing $u$ stops at $u$. For any vertex that has passed the inspection, we continue to inspect its own offspring. Note that we do not inspect the root: $\theta$ is considered passing the inspection whether or not $T_\theta\le \varepsilon$. In this manner, the inspection process might keep going indefinitely or stop after finitely many steps. To show that the process stops almost surely after finite number of inspections we establish the following more general result.

\begin{lem}\label{24201}
Assume there exists $0<\nu<\infty$ such that 
\[\max\left\{1,\,\mathbb{E}K_\theta,\,\sup_{v\in\mathcal{T}\backslash\{\theta\}}\mathbb{E}\left[{{K}_{v }}|\ \{{K}_{u }\}_{|u|\le |v |-1}\right]\right\}< \nu<\infty.\]
Let $0<\varepsilon < \log\frac{\nu}{\nu-1}.$ Then 
the inspection process 
almost surely stops after finitely many steps and $B=(\varepsilon, \infty)$ is a cut-set recurrent with respect to $(\{T_v\}_{v\in\mathbb{V}}, \mathcal{T}).$
\end{lem}

\begin{proof}
The probability for a vertex $v\in\mathcal{T}$ to pass the inspection is
\[\delta:=\mathbb{P}({T}_v\le\varepsilon)=\int_0^\varepsilon e^{-t}dt=1-e^{-\varepsilon}<\frac{1}{\nu}.\]
The number of offspring of the root $\theta$ that pass the inspection is 
\[{{\tilde{K}}_{\theta}}=\sum\limits_{j=1}^{{{K}_{\theta }}}{{\mathbbm{1}_{{{T}_{ j}}\le \varepsilon }}}\]
which is the sum of $K_\theta$ i.i.d.\ random variables with distribution Bernoulli($\delta$) that are also independent of $K_\theta$. 
Thus
\[\mathbb{E}\tilde{K}_\theta=
\mathbb{E}[\delta{K}_\theta]\le\delta\nu<1.\]
For each $n\ge 0$, denote by $V_n$ be the total number of vertices of the $n$'th generation that pass the inspection. We have $V_0=1$ and $V_1=\tilde{K}_\theta$. If we label the vertices of the $n$'th generation that pass the inspection by $v_{n,j}$ for $1\le j\le V_n$ then
\[{{V}_{n+1}}=\sum\limits_{j=1}^{{{V}_{n}}}{{{{\tilde{K}}}_{{{v }_{n,j}}}}}\ \ \ \forall\, n\ge 0.\]
Write $u=v_{n,j}$. Since $\{T_v\}_{v\in\mathbb{V}}$ are independent from $\{K_v\}_{v\in\mathbb{V}}$, we have
\[\mathbb{E}\left[ {{{\tilde{K}}}_{u }}|{{V}_{n}} \right] = \mathbb{E}\left[ \sum\limits_{j=1}^{{{K}_{u}}}{{\mathbbm{1}_{{{T}_{u * j}}\le \varepsilon }}|\,{{V}_{n}}} \right] = \delta \mathbb{E}\left[ {{K}_{u }}|{{V}_{n}} \right] \]
Moreover, 
\begin{eqnarray*}
\mathbb{E}\left[ {{K}_{u }}|{{V}_{n}} \right]
&=&\mathbb{E}\left[ \mathbb{E}[{{K}_{u}}|{{\{{{K}_{v}}\}}_{|v|\le n}}]|{{V}_{n}} \right]\\
&\le& \mathbb{E}\left[ \nu |{{V}_{n}} \right]=\nu .
\end{eqnarray*}
Thus
\[\mathbb{E}{{V}_{n+1}}=  \mathbb{E}\left[\mathbb{E}\left[ {{{\tilde{K}}}_{w }}|{{V}_{n}} \right] V_n\right]
\le 
\delta \nu \mathbb{E}{{V}_{n}}.\]
Because $\delta\nu<1$, $\lim_{n\to\infty}\mathbb{E}V_{n}=0$. Put $V_\infty=\liminf V_{n}$. By Fatou's lemma, $\mathbb{E}V_\infty=0$. Therefore, $V_\infty=0$ almost surely 
and hence, 
the inspection process for $\{T_v\}_{v\in\mathbb{V}}$ almost surely stops in finitely many steps.

To show that the interval $B=(\varepsilon,\infty)$ is cut-set recurrent  with respect to $(\{T_v\}_{v\in\mathbb{V}},\mathcal{T})$,
construct the sequence of passage cut-sets $\tilde{\Pi}_B^{(k)}$ by the procedure described in the proof of \autoref{cutest-red-thm}: The first passage cut-set $\tilde{\Pi}_B^{(1)}$ consists of, for each $s\in\partial\mathcal{T}$, the first vertex $v\neq\theta$ such that $T_{v}>\varepsilon$. Starting at each vertex in $\tilde{\Pi}_B^{(k)}$, we start a new independent inspection process. The union of the vertices where one of these inspection processes stops gives the passage cut-set $\tilde{\Pi}_B^{(k+1)}$.  Since each of these sets is finite, $B$ is cut-set recurrent as claimed.
\end{proof}

Define a strictly increasing sequence of random variables $\tilde{H}_k$ as: $\tilde{H}_0=1$ and
\[\tilde{H}_{k}=\max\{|v|:~v\in\tilde{\Pi}_B^{(k)}\}~~~\forall\,k\in\mathbb{N}.\]
It then follows that almost surely, 
for each $k\in\mathbb{N}$ and $s\in\partial\mathcal{T}$, there exists $\tilde{l}_k(s)\in[k,\, \tilde{H}_k)$ such that $s|{\tilde{l}_k(s)}\in\tilde{\Pi}_B^{(k)}$. 
Then, for $n>\tilde{H}_k$
$$
\tilde{\zeta}_n = \underset{\begin{smallmatrix} 
 v \in \mathcal{T} \\ 
 |v |=n 
\end{smallmatrix}}{\mathop{\min }}
\sum\limits_{k=1}^{n}{{{{{T}}}_{v |k}}}
\ge \underset{\begin{smallmatrix} 
 v \in \mathcal{T} \\ 
 |v |=n  
\end{smallmatrix}}{\mathop{\min }}\,\sum\limits_{i=1}^{k}{{{{{T}}}_{{\tilde{l}_{i}}}}}\ge k\varepsilon
$$
where 
$\tilde{l}_i=\tilde{l}_i(s_v)$ with $s_v\in\partial\mathcal{T}$ such that $s_v|_{|v|}=v$.
Letting $k\to\infty$, we get $\tilde{\zeta}=\lim{\tilde{\zeta}}_n=\infty$ which implies \eqref{eq1:11420}.  This completes the proof of \autoref{main}.


\section{Criteria for cut-set recurrence}\label{corproof}

In this section, we give two criteria for a branching Markov chain to satisfy the condition (E) in terms of the transition probability density $p(x,y)$. These criteria will be used to show the non-explosion of the cascade of the KPP equation on the Fourier side (\autoref{kppcascade}) and of the Bessel cascade of the three-dimensional Navier-Stokes equations (\autoref{besselcascade}).
\begin{prop}\label{criterion1}
Let $\{X_v\}_{v\in\mathbb{T}}$ be a branching Markov chain on the state space $S$ satisfying \textup{(A)} and \textup{(B)} with the transition probability density $p(x,y)$. Suppose there exist a constant $b>0$ and functions $\phi_1,\phi_2:S\to(0,\infty)$ such that:
\begin{enumerate}[(i)]
\item $p(x,y)\le\phi_1(x)\phi_2(y)$ for all $x,y\in S_b=\lambda^{-1}((b,\infty))$.
\item $\phi_1=\phi\circ\lambda$ for some locally bounded function $\phi:(0,\infty)\to(0,\infty)$.
\item $\phi_2$, $\phi_1\phi_2\in L^1(S_b)$.
\end{enumerate}
Then $\{X_v\}_{v\in\mathbb{T}}$ satisfies the condition \textup{(E)}.
\end{prop}
\begin{proof}
By \autoref{main}, it is sufficient to show that there exist constants $c>b$ and $r>2$ such that ${I}_n\lesssim\phi_1(a)r^{-n}$, where ${I}_n$ is defined by 
\[I_n(a,c)=\mathbb{P}_a(X_{s|1},\ldots X_{s|n}\in S_c)\ \ \ \forall a\in S,\,n\in \mathbb{N}\]
with $S_c=\lambda^{-1}((c,\infty))$.
Fix an initial state $X_\theta=a>0$ and a path $s\in \partial\mathbb{T}$. 
For simplicity, we write the sequence $X_{\theta}$, $X_{s|1}$, $X_{s|2}$,\ldots as $X_0$, $X_1$, $X_2$,\ldots Since this is a time homogeneous Markov chain with transition probability density $p$,
\[{I}_n={\int_{S_c}{\ldots \int_{S_c}p(a,y_1)p(y_1,y_2)\ldots p(y_{n-1},y_n)d{{y}_{n}}\ldots d{{y}_{1}}}}~~~\forall\,n\ge 1.\]
By the condition (i), $p(y_{j-1},y_j)\le\phi_1(y_{j-1})\phi_2(y_j)$. We obtain the estimate
\[{{I}_{n}}\le {{\phi }_{1}}(a){{\left( \int_{S_c}{{{\phi }_{1}}{{\phi }_{2}}dx} \right)}^{n-1}}\int_{S_c}{{{\phi }_{2}}dx}\le {{\phi }_{1}}(a)\frac{{{\left\| {{\phi }_{2}} \right\|}_{{{L}^{1}}(S_b)}}}{{{\left\| {{\phi }_{1}}{{\phi }_{2}} \right\|}_{{{L}^{1}}(S_b)}}}{{r}^{-n}}\]
where $1/r=\int_{S_c}{{{\phi }_{1}}{{\phi }_{2}}dx}$. By choosing $c>b$ sufficiently large, we get $r>2$.
\end{proof}
\begin{prop}\label{criterion}
Let
$\{X_v\}_{v\in\mathbb{T}}$ be a branching Markov chain on the state space $S=(0,\infty)$ satisfying \textup{(A), (B)} and with following conditions:
\begin{enumerate}[(i)]
\item For any path $s\in\partial\mathbb{T}$, the Markov chain $X_\theta$, $X_{s|1}$, $X_{s|2}$,\ldots is time-reversible with the transition probability density $p(x,y)$ and the invariant probability density $\gamma(x)$. 
\item There exist $c_1>1$ and a locally bounded function $\psi_1:(0,\infty)\to\mathbb{R}$ such that
\[p(x,y)\le \psi_1(x)\gamma(y)~~~~\forall\, x,y>c_1.\]
\item There exists $c_2>1$ such that
\[p(x,y)\le \frac{c_2}{y}\ \ \ \forall\, c_1<y<x.\]
\item There exists a function $\psi_2:(0,\infty)\to\mathbb{R}$ such that 
\[p(x,y)\le \frac{\psi_2(y-x)}{x}\ \ \ \forall\, c_1<x<y.\]
\item There exists $\alpha>2c_2$ such that
\[\int_0^\infty \psi_2(x)x^\alpha dx<\infty,~~~\int_0^\infty \gamma(x)x^\alpha dx<\infty.\]
\end{enumerate}
Then $\{X_v\}_{v\in\mathbb{T}}$ satisfies the condition \textup{(E)} with $A=(0,c]$, $2<r<\alpha/c_2$, $\psi=2\|\gamma x^\alpha\|_{L^1}\psi_1$ and
\[c=\max\left\{c_1,\,\frac{2^{\alpha-1} r(\|\psi_2\|_{L^1}+\|\psi_2 x^\alpha\|_{L^1})}{1-(c_2r)/\alpha}\right\}.\]
\end{prop}

\begin{proof}
As noted in \autoref{X>0}, it is sufficient to show ${I}_n\le\psi(a)r^{-n}$, where ${I}_n$ is defined by \eqref{I_n'}.
Fix an initial state $X_\theta=a>0$ and a path $s\in \partial\mathbb{T}$. 
Again, we write the sequence $X_{\theta}$, $X_{s|1}$, $X_{s|2}$,\ldots as $X_0$, $X_1$, $X_2$,\ldots Since this is a time homogeneous Markov chain with transition probability density $p$,
\[{I}_n={\int_{c}^{\infty }{\ldots \int_{c}^{\infty }p(a,y_1)p(y_1,y_2)\ldots p(y_{n-1},y_n)d{{y}_{n}}\ldots d{{y}_{1}}}}~~~\forall\,n\ge 1.\]
Consider a sequence of functions $\{f_n\}$ given by $f_1(y)=p(a,y)$ and 
\[{{f}_{n+1}}(x)=\int_{c}^{\infty }{p(y,x){{f}_{n}}(y)dy}\ \ \ \forall n\ge 1.\]
Then ${I}_n=\int_{c}^\infty f_n(y)dy$. Define a sequence of functions $\{g_n\}$ by $g_1(x)\equiv 1$ and
\begin{equation}\label{eq:918202}
g_{n+1}(x)=\int_{c}^\infty p(x,y)g_n(y)dy~~~\forall n\ge 1.
\end{equation}
We now show by induction that $f_n(x)\le \psi_1(a)\gamma(x)g_n(x)$. This is true for $n=1$ thanks to condition (ii). Suppose $f_n\le \psi_1(a)\gamma g_n$ for some $n\ge 1$. By the induction hypothesis and condition (i),
\begin{eqnarray*}
{{f}_{n+1}}(x)&=&\int_{c}^{\infty }{p(y,x){{f}_{n}}(y)dy}\le \psi(a)\int_{c}^{\infty }{p(y,x)\gamma (y){{g}_{n}}(y)dy}\\
&=&\psi_1(a)\int_{c}^{\infty }{p(x,y)\gamma (x){{g}_{n}}(y)dy}=\psi(a)\gamma (x){{g}_{n+1}}(x)
\end{eqnarray*}
which completes the proof by induction. 
We show by induction that
\begin{equation}\label{eq:918201}
g_n(x)\le\min\{x^\alpha r^{-n},\,1\}~~~\forall n\in\mathbb{N},\, x>c.
\end{equation}
Using the fact that $\int_0^\infty p(x,y)dy=1$, we can deduce from \eqref{eq:918202} that $g_n(x)\le 1$ for all $n\in\mathbb{N}$ and $x>0$. Because $c>1$, \eqref{eq:918201} is true for $n=1$. Suppose by induction that for some $n\ge 1$, \[g_n(x)\le x^\alpha r^{-n}~~~\forall x>c.\]
We decompose $g_{n+1}$ given by \eqref{eq:918202} as follows.
\[{{g}_{n+1}}(x)=\underbrace{\int_{c}^{x}{p}(x,y){{g}_{n}}(y)dy}_{\{1\}}+\underbrace{\int_{x}^{\infty }{p}(x,y){{g}_{n}}(y)dy}_{\{2\}}.\]
The first term can be estimated using condition (c):
\[\{1\}\le \int_{c}^{x}{\frac{{{c}_{2}}}{y}({{y}^{\alpha }}{{r}^{-n}})dy}={{c}_{2}}{{r}^{-n}}\int_{c}^{x}{{{y}^{\alpha -1}}dy}<{{x}^{\alpha }}{{r}^{-n-1}}\kappa .\]
The second term can be estimated using condition (d):
\[\{2\}\le \int_{x}^{\infty }{\frac{\psi_2(y-x)}{x}({{y}^{\alpha }}{{r}^{-n}})dy}=\frac{{{r}^{-n}}}{x}\int_{0}^{\infty }{\psi_2(z){{(x+z)}^{\alpha }}dz}.\]
By the inequality $(x+z)^\alpha\le 2^{\alpha-1}(x^\alpha+z^\alpha)$, we have
\begin{eqnarray*}
\{2\}&<&\frac{{{r}^{-n}}{{2}^{\alpha -1}}}{c}\int_{0}^{\infty }{\psi_2(z)({{x}^{\alpha }}+{{z}^{\alpha }})dz}\\
&<&{{x}^{\alpha }}{{r}^{-n-1}}\left( \frac{{{2}^{\alpha -1}}r{{\left\| \psi_2 \right\|}_{{{L}^{1}}}}+{{2}^{\alpha -1}}r{{\left\| \psi_2{{z}^{\alpha }} \right\|}_{{{L}^{1}}}}}{c} \right)\\
&<&{{x}^{\alpha }}{{r}^{-n-1}}(1-\kappa )
\end{eqnarray*}
where $\kappa=\frac{c_2r}{\alpha}$. By the above estimates, $g_{n+1}(x)=\{1\}+\{2\}<x^\alpha r^{-n-1}$. Therefore, \eqref{eq:918201} is true for every $n$. We proceed to estimate ${I}_n$. Put $\beta_n=r^{n/\alpha}$. If $\beta_n>c$ then
\begin{eqnarray*}
\frac{{{{I}}_{n}}}{\psi_1 (a)}=\frac{1}{\psi_1 (a)}\int_{c}^{\infty }{{{f}_{n}}(x)dx}&\le& \int_{c}^{\infty }{\gamma (x){{g}_{n}}(x)dx}\\
&\le& \underbrace{\int_{c}^{{{\beta }_{n}}}{\gamma (x){{x}^{\alpha}}{{r}^{-n}}dx}}_{\{3\}}+\underbrace{\int_{{{\beta }_{n}}}^{\infty }{\gamma (x)dx}}_{\{4\}}.
\end{eqnarray*}
Each term can be estimated as follows:
\begin{eqnarray*}
\{3\}&\le& {{r}^{-n}}\int_{0}^{\infty }{\gamma (x){{x}^{\alpha }}dx}={{r}^{-n}}{{\left\| \gamma {{x}^{\alpha }} \right\|}_{{{L}^{1}}}},\\
\{4\}&\le& \beta _{n}^{-\alpha }\int_{{{\beta }_{n}}}^{\infty }{\gamma (x){{x}^{\alpha }}dx}\le{{r}^{-n}}{{\left\| \gamma {{x}^{\alpha }} \right\|}_{{{L}^{1}}}}.
\end{eqnarray*}
Therefore, ${I}_n\le 2\|\gamma x^\alpha\|_{L^1}\psi_1(a)r^{-n}$. If $\beta_n\le c$ then
\[\frac{{I}_n}{\psi_1(a)}\le \{4\}\le r^{-n}\|\gamma x^\alpha\|_{L_1}.\]
In this case, we have ${I}_n\le \|\gamma x^\alpha\|_{L^1}\psi_1(a)r^{-n}$.
\end{proof}
\section{Self-similar DSY cascades and proof of \autoref{multibran}}
\label{selfsimilarsection}

We begin this section noting some properties of self-similar DSY cascades.
\begin{lem}\label{selfsimilar}
Let $\{\lambda(X_v)^{-1}T_v\}_{v\in\mathbb{T}}$ be a self-similar DSY cascade on $S\subset(0,\infty)$. Then for any initial state $X_\theta=a\in S$ and path $s\in\partial\mathbb{T}$, the ratios $\{R_{s|j}=X_{s|j}/X_{s|j-1}\}_{j\ge 1}$ form an i.i.d.\ sequence with the common distribution $p(1,dx)$.
\end{lem}
\begin{proof}
For $v\in\mathbb{T}\backslash\{\theta\}$, the ratio $R_v$ can be written as $R_v=X_v/X_{\overleftarrow{v}}$, where $\overleftarrow{v}$ denotes the parent of $v$. Now fix $v\in\mathbb{T}\backslash\{\theta\}$. By the definition of self-similar DSY cascades, the family $\{Y_w=X_{\overleftarrow{v}*w}/X_{\overleftarrow{v}}\}_{w\in\mathbb{T}}$ satisfies (A), (B) with the transition probability $p(x,dy)$. Since $Y_\theta=1$, $Y_1$ and $Y_2$ have the distribution $p(1,dx)$. Therefore, $R_v$ also has the distribution $p(1,dx)$. 

Next, we show that the transition probability has the scaling property $p(a,adx)=p(1,dx)$ for all $a\in S$. This is obtained by expressing the cumulative distribution functions of $R_1$ in two ways. First, $\mathbb{P}_a(R_1\le x)=\int_0^xp(1,dz)$ for any $a,x\in S$. Second, $\mathbb{P}_a(R_1\le x)=\mathbb{P}_a(X_1\le ax)=\int_0^{ax}p(a,dy)$. By the change of variables $z= y/a$ in the last integral, one obtains $p(a,adz)=p(1,dz)$.

Finally, we show that along each path $s\in\partial\mathbb{T}$, the ratios $\{R_{s|j}\}_{j\ge 1}$ are independent. For any $r_1,r_2,\ldots,r_n\in S$,
\begin{eqnarray*}
{{\mathbb{P}}_{a}}({{R}_{1}}\le {{r}_{1}},\ldots ,{{R}_{n}}\le {{r}_{n}})&=&{{\mathbb{P}}_{a}}({{X}_{1}}\le a{{r}_{1}},{{X}_{2}}\le {{X}_{1}}{{r}_{2}},\ldots ,{{X}_{n}}\le {{X}_{n-1}}{{r}_{n}})\\
&=&\int_{0}^{a{{r}_{1}}}{\int_{0}^{{{x}_{1}}{{r}_{2}}}{\ldots \int_{0}^{{{x}_{n-1}}{{r}_{n}}}{p(a,d{{x}_{1}})p({{x}_{1}},d{{x}_{2}})\ldots p({{x}_{n-1}},d{{x}_{n}})}}}
\end{eqnarray*}
where $R_j$ and $X_j$ are short notations for $R_{s|j}$ and $X_{s|j}$. By the change of variables $y_j=x_{j}/x_{j-1}$ (with $x_0=a$) and the scaling property of $p$, one has
\begin{eqnarray*}
{{\mathbb{P}}_{a}}({{R}_{1}}\le {{r}_{1}},\ldots ,{{R}_{n}}\le {{r}_{n}})&=&\int_{0}^{{{r}_{1}}}{\int_{0}^{{{r}_{2}}}{\ldots \int_{0}^{{{r}_{n}}}{p(1,d{{y}_{1}})p(1,d{{y}_{2}})\ldots p(1,d{{y}_{n}})}}}\\
&=&\mathbb{P}({{R}_{1}}\le {{r}_{1}})\ldots \mathbb{P}({{R}_{n}}\le {{r}_{n}}).
\end{eqnarray*}
\end{proof}
We now give the proof of \autoref{multibran}.
By \autoref{main} (taking into account \autoref{X>0}),
 we only need to find $c>0$, $r>2$, and a locally bounded function $\psi$ on $(0,\infty)$ such that
\[I_n(a,c)=\mathbb{P}_a(X_{s|1}>c,\,X_{s|2}>c,\ldots,X_{s|n}>c)\le \psi(a)r^{-n}\ \ \ \forall\, n\ge 1\]
where $s\in \partial\mathbb{T}$ is an arbitrary path. For simplicity, we write the sequences $\{X_{s|j}\}_{j\ge 0}$ and $\{R_{s|j}\}_{j\ge 1}$ as $\{X_j\}_{j\ge 0}$ and  $\{R_j\}_{j\ge 1}$, respectively. Recall that $p(1,dx)$ is the distribution of $R_1$. Define $I_0=1$. For $n\ge 1$, we have
\begin{eqnarray*}
{{I}_{n}}(a,c)&=&{{\mathbb{P}}_{a}}\left( a{{R}_{1}}>c,\ldots,a{{R}_{1}}{{R}_{2}}...{{R}_{n}}>c \right)\\
&=&{{\mathbb{P}}_{a}}\left( {{R}_{1}}>\frac{c}{a},a{{R}_{2}}>\frac{c}{{{R}_{1}}},\ldots,a{{R}_{2}}...{{R}_{n}}>\frac{c}{{{R}_{1}}} \right)\\
&=&\int_{c/a}^{\infty }{{{\mathbb{P}}_{a}}\left( a{{R}_{2}}>\frac{c}{x},\ldots,a{{R}_{2}}...{{R}_{n}}>\frac{c}{x} \right)p(1,dx)}\\
&=&\int_{c/a}^{\infty }{{{I}_{n-1}}(a,c/x)p(1,dx)}.
\end{eqnarray*}
Therefore,
\[{{I}_{n}}(a,t)=\int_{t/a}^{\infty }{{{I}_{n-1}}(a,t/x)p(1,dx)}\ \ \ \forall a,t\in S.\]
Next, we show by induction that
\begin{equation}\label{36211}
I_n(a,t)\le\left(\frac{a}{t}\right)^br^{-n}\ \ \ \forall a,t\in S,\,n\ge 1
\end{equation}
where $r=1/\mathbb{E}[R_1^b]>2$. For $n=1$,
\[{{I}_{1}}(a,t)=\int_{t/a}^{\infty }{p(1,dx)}\le {{\left( \frac{a}{t} \right)}^{b}}\int_{t/a}^{\infty }{{{x}^{b}}p(1,dx)}\le {{\left( \frac{a}{t} \right)}^{b}}\mathbb{E}[{R_1^{b}}]={{\left( \frac{a}{t} \right)}^{b}}{{r}^{-1}}.\]
Suppose \eqref{36211} is true for some $n\ge 1$. Then
\[{{I}_{n+1}}(a,t)\le \int_{t/a}^{\infty }{{{\left( \frac{ax}{t} \right)}^{b}}{{r}^{-n}}p(1,dx)}\le {{\left( \frac{a}{t} \right)}^{b}}{{r}^{-n}}\mathbb{E}[{R_1^{b}}]={{\left( \frac{a}{t} \right)}^{b}}{{r}^{-n-1}}.\]
Therefore, \eqref{36211} is also true for $n+1$. We can now choose $c=1$ and $\psi(a)=a^b$.

\section{Proof of \autoref{expcriterion}}
\label{explosioncriteria}
Starting from the root $\theta$, we construct a random path in $\partial\mathbb{T}$ by recursively annexing one vertex at a time as follows. Suppose $v\in\mathbb{T}$ is the most recently annexed vertex. If $\lambda(X_{v*1})\ge \lambda(X_{v*2})$ then $v*1$ is the next to be annexed. Otherwise, $v*2$ is the next to be annexed. This random path, which we denote by $s$, is a path with stepwise maximal intensities: at every branching step, the path follows the branch that has a larger intensity.
\begin{figure}[h]
\centering
\includegraphics[scale=1]{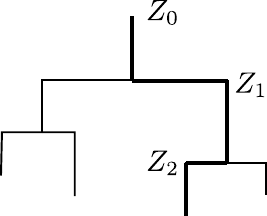}
\caption{Path with stepwise maximal intensities.}
\label{maximal}
\end{figure}
For $n\ge 0$, let $Z_n=\lambda(X_{s|n})$. Then
\[\zeta \le {{\zeta }_{*}}:=\sum\limits_{j=0}^{\infty }{\frac{{{T}_{s|j}}}{\lambda ({{X}_{s|j}})}}=\sum\limits_{j=0}^{\infty }{\frac{{{T}_{s|j}}}{{{Z}_{j}}}}.\]
By the independence of the clocks and the intensities,
\begin{equation}\label{316211}{{\mathbb{E}}_{a}}\zeta =\sum\limits_{j=0}^{\infty }{\mathbb{E}_a[ Z_{j}^{-1} ]}.\end{equation}
We only need to show that
\begin{equation}\label{316212}
\mathbb{E}_a[Z_j^{-1}]\le \kappa^j\lambda(a)^{-1}\ \ \ \forall\,j\in\mathbb{N},\,a\in S.
\end{equation}
Indeed, once \eqref{316212} is proved, one can infer from \eqref{316211} that
\[{{\mathbb{E}}_{a}}\zeta \le \sum\limits_{j=0}^{\infty }{\frac{{{\kappa}^{j}}}{\lambda (a)}}=\frac{1}{\lambda (a)}\frac{1}{1-\kappa}<\infty \]
which leads to $\zeta<\infty$ a.s. Next, to show \eqref{316212}, it suffices to show that
\begin{equation}\label{316213}
\mathbb{E}_a[Z_{j+1}^{-1}]\le \kappa\mathbb{E}_a[Z_j^{-1}]\ \ \ \forall\,j\ge 0,\,a\in S.
\end{equation}
Fix $j\ge 0$ and put $v=s|j\in\mathbb{T}$. Then $Z_j=\lambda(X_v)$ and $Z_{j+1}=\max\{\lambda(X_{v*1}),\lambda(X_{v*2})\}$. Because the joint distribution of $(X_{v*1},x_{v*2})$ given $X_v$ is the same as the joint distribution of $(X_1,X_2)$ given $X_\theta$,
\begin{eqnarray*}
\mathbb{E}_a[Z_{j+1}^{-1}|{{X}_{v}}=x]&=&\mathbb{E}_a\left[ \frac{1}{\max \left\{ \lambda ({{X}_{v*1}}),\lambda ({{X}_{v*2}}) \right\}}|{{X}_{v}}=x \right]\\
&=&\mathbb{E}\left[ \frac{1}{\max \left\{ \lambda ({{X}_{1}}),\lambda ({{X}_{2}}) \right\}}|{{X}_{\theta }}=x \right]\\
&=&\mathbb{E}_x[Z_{1}^{-1}]\le \kappa\lambda(x)^{-1}.
\end{eqnarray*}
Thus, $\mathbb{E}[Z_{j+1}^{-1}|X_v]\le \kappa\lambda(X_v)^{-1}=\kappa Z_j^{-1}$. The proof of \eqref{316213} is then completed by the law of total expectation,
\[\mathbb{E}_a[Z_{j+1}^{-1}]=\mathbb{E}_a[\mathbb{E}[Z_{j+1}^{-1}|{{X}_{v}}]]\le \kappa\mathbb{E}_a[Z_{j}^{-1}].\]

\section{Proof of \autoref{01}}
\label{zeroone}
With the initial state $X_\theta=x\in S$, the explosion time can be expressed as $\zeta =T_\theta\lambda(x)^{-1}+\min \left\{ {{\zeta }^{(1)}},{{\zeta }^{(2)}} \right\}$ where $\zeta^{(k)}$ is the explosion time of the sub-cascade starting at vertex $k\in\{1,2\}$.
\begin{eqnarray}
\nonumber f(x):={{\mathbb{P}}_{x}}(\zeta =\infty )&=&{{\mathbb{E}}_{x}}[\mathbb{E}[{\mathbbm{1}_{{{\zeta }^{(1)}}=\infty }}{\mathbbm{1}_{{{\zeta }^{(2)}}=\infty }}|{{X}_{1}},{{X}_{2}}]]\\
\nonumber &=&{{\mathbb{E}}_{x}}[\mathbb{E}[{\mathbbm{1}_{{{\zeta }^{(1)}}=\infty }}|{{X}_{1}}]\mathbb{E}[{\mathbbm{1}_{{{\zeta }^{(2)}}=\infty }}|{{X}_{2}}]]\\
\label{eq:611211}&=&{{\mathbb{E}}_{x}}[f(X_1)f(X_2)].
\end{eqnarray}
By \autoref{selfsimilar}, we can express $X_v$, $v\in\mathbb{T}$, as a product of i.i.d.\ ratios $X_v=xR_{v|1}\ldots R_{v||v|}$. Let us first assume that condition (i) is satisfied. We use the multiplicativity of $\lambda$ to rewrite the explosion time as $\zeta=\lambda(x)^{-1}\widetilde{\zeta}$ where
\[\widetilde{\zeta} =\underset{n}{\mathop{\sup }}\,\underset{|v|=n}{\mathop{\min }}\,\sum\limits_{j=0}^{n}{\frac{{{T}_{v|j}}}{\lambda ({{R}_{v|1}})...\lambda ({{R}_{v|n}})}}.\]
The event $[\zeta=\infty]$ is the same as the event $[\widetilde{\zeta}=\infty]$, whose probability does not depend on $x$. Thus, $f(x)\equiv c$. By \eqref{eq:611211}, $c=c^2$. Therefore, $c\in\{0,1\}$.

Next, assume that condition (iii) is satisfied. From \eqref{eq:611211}, one has the estimate $2f(x)\le \mathbb{E}_x[f(X_1)^2]+\mathbb{E}_x[f(X_2)^2]$. Now note that along each path $s\in\partial\mathbb{T}$, $\gamma$ is a stationary distribution of the Markov chain $\{X_{s|j}\}_{j\ge 0}$. By integrating the above inequality against $\gamma(dx)$, one obtains 
 \[2\mathbb{E}_\gamma f({{X}_{\theta }})\le \mathbb{E}_\gamma[f{{({{X}_{1}})}^{2}}]+\mathbb{E}_\gamma[f{{({{X}_{2}})}^{2}}]=2\mathbb{E}_\gamma[f{{({{X}_{\theta }})}^{2}}]\]
which implies $f(x)\in\{0,1\}$ for $\gamma$-a.e.\ $x\in S$. Consider two cases.\\
$\blacksquare$ $f(a)=1$ for some $a\in S$:\\
Using the inequality $f(X_1)f(X_2)\le f(X_1)$, one has $1=f(a)\le {{\mathbb{E}}_{a}}[f({{X}_{1}})]=\int_{S}{f(x)p(a,dx)}$. Thus, $f(x)=1$ for $p(a,dx)$-a.e.\ $x\in S$. Because $\gamma(dx)\ll p(a,dx)$, $f(x)=1$ for $\gamma$-a.e.\ $x\in S$. Now take an arbitrary $b\in S$. Since $p(b,dx)\ll\gamma(dx)$, $f(x)=1$ for $p(b,dx)$-a.e.\ $x\in S$. By the inequality $(f(X_1)-1)(f(X_2)-1)\ge 0$, we have
\[f(b)={{\mathbb{E}}_{b}}[f({{X}_{1}})f({{X}_{2}})]\ge {{\mathbb{E}}_{b}}[f({{X}_{1}})]+{{\mathbb{E}}_{b}}[f({{X}_{2}})]-1=2\int_{S}{f(x)p(b,dx)}-1=1.\]
Therefore, $f(b)=1$.\\
$\blacksquare$ $f(x)<1$ for all $x\in S$:\\
In this case, $f(x)=0$ for $\gamma$-a.e.\ $x\in S$. For any $b\in S$, $f(x)=0$ for $p(b,dx)$-a.e.\ $x\in S$. By the inequality $f(X_1)f(X_2)\le f(X_1)$, we have
\[f(b)={{\mathbb{E}}_{b}}[f({{X}_{1}})f({{X}_{2}})]\le {{\mathbb{E}}_{b}}[f({{X}_{1}})]=\int_{S}{f(x)p(b,dx)}=0.\] 
Therefore, $f(b)=0$.

Finally, assume that condition (ii) is satisfied. Using the same technique as in Part (iii) and noting that $\gamma$ is fully supported on $S$ in the discrete topology, one has $f(x)\in\{0,1\}$ for all $x\in S$. Let $A=\{x\in S:\,f(x)=1\}$. If $A=\emptyset$ then $f(x)\equiv 0$. Otherwise, for any $a\in A$, we have $1=f(a)\le\mathbb{E}_a[f(X_1)]=p(a,A)$, which leads to $p(a,A)=1$. This implies that $A$ contains all the states that can be reached from an element of $A$ in one step. By the irreducibility of the Markov chain, $A=S$.
\section{DSY cascades resulting from probabilistic models}\label{prob_examples}
\label{examples}
In this section, we give four examples of non-explosive cascades motivated by probabilistic considerations. 

\subsection{A pure birth process}\label{yuleprocess}
The classical Yule process is a pure process starting from a single progenitor in which a particle survives for a mean $\lambda^{-1}$ (deterministic constant) exponentially distributed time before
being replaced by two offspring independently evolving in the same manner. 
The population is finite at every finite time (i.e.\ non-explosive) if and only if
\begin{equation}\label{eq:926203}
\zeta =\underset{n\to \infty }{\mathop{\lim }}\,\underset{|v|=n}{\mathop{\min }}\,\frac{1}{\lambda}\sum\limits_{j=0}^{n}T_{v|j}=\infty~~~\text{a.s.}
\end{equation}
The classical Yule cascades are the simplest case of DSY cascades where $\lambda(x)\equiv \lambda$. The branching Markov chain can be chosen arbitrarily. We can of course choose $X_v\equiv 1$ (deterministic). It is well-known that the classical Yule cascades are non-explosive. One can apply \autoref{main} with the choice $c=2\lambda$ to obtain an alternative proof by cut-set theory. In fact, \autoref{24201} implies the non-explosion of a more general pure birth process:
\begin{prop}\label{purebirth}
Let $\mu$ be a positive number. Consider a branching process starting with a single progenitor in which a particle, independently of all others, survives for a mean-one exponentially distributed time before being replaced by a random number of offspring. The offspring distributions are not necessarily the same across all the particles. Suppose that the expected number of offspring of each particle is less than $\mu$. Then the total population is finite at any finite time.
\end{prop}
In \autoref{purebirth}, if the offspring distributions are the same across all the particles then the pure birth process  is a continuous-time Galton-Watson process in which the offspring distribution has a finite expected value. The non-explosion problem for the case of infinite expected offspring number was studied in \cite{amini2013}.


\subsection{A mean-field cascade (dependent first passage percolation on a tree)}
Consider a DSY cascade $\{\lambda(X_v)^{-1}T_v\}_{v\in\mathbb{T}}$ on the state space $S=(0,\infty)$ such that along each path $s\in\partial\mathbb{T}$, $\{X_{s|j}\}_{j\ge 1}$  is an i.i.d.\ sequence of random variables. The branching Markov chain $\{X_v\}_{v\in\mathbb{T}}$ clearly has properties (A), (B), (C), (D). Choose $c>0$ sufficiently large such that $\mathbb{P}(X_1>c)<1/2$. Then
\[I_n(a,c)=\mathbb{P}_a(X_{s|1}>c,X_{s|2}>c,\ldots,X_{s|n}>c)=(1/\mathbb{P}(X_1>c))^{-n}.\]
Therefore, the condition (E) holds for $\psi\equiv 1$ and $r=1/\mathbb{P}(X_1>c)$. By \autoref{main}, the cascade is non-explosive.

\subsection{A cascade with geometric-like sequence along each path}
Consider a DSY cascade $\{\lambda(X_v)^{-1}T_v\}_{v\in\mathbb{T}}$ on the state space $S=(0,\infty)$ in which the branching Markov chain has properties (A), (B), (C) and the transition probability density is given by
\[p(x,y)=\frac{e^{-y}}{1-e^{-2x}}\mathbbm{1}_{0<y<2x}.\]
Intuitively, $X_{v*1}$ can be as large as $2X_v$ but with a small probability. This allows the sequence $\{X_{s|j}\}_{j\ge 0}$ to behave like a geometric sequence up to any prescribed index. It is the geometric growth of intensities along each path that causes the explosion in the $\alpha$-Riccati equation for $\alpha>1$ (\autoref{alpha-ric}). We show that the present cascade is in fact non-explosive. Denote
\[{{I}_{n}}(a)={{\mathbb{P}}_{a}}({{X}_{s|1}},{{X}_{s|2}},\ldots ,{{X}_{s|n}}>1)=\int_{1}^{\infty }{\ldots \int_{1}^{\infty }{p}(a,{{y}_{1}})p({{y}_{1}},{{y}_{2}})\ldots p({{y}_{n-1}},{{y}_{n}})d{{y}_{n}}\ldots d{{y}_{1}}}.\]
If $a\le 1/2$ then $p(a,y_1)=0$ for all $y_1>1\ge 2a$. In this case, $I_n(a)=0$. If $a>1/2$ then
\[{{I}_{n}}(a)=\int_{1}^{2a}{\ldots \int_{1}^{2{{y}_{n-1}}}{\frac{{{e}^{-{{y}_{1}}}}}{1-{{e}^{-2a}}}\frac{{{e}^{-{{y}_{2}}}}}{1-{{e}^{-2{{y}_{1}}}}}\ldots \frac{{{e}^{-{{y}_{n}}}}}{1-{{e}^{-2{{y}_{n-1}}}}}d{{y}_{n}}\ldots d{{y}_{1}}}}.\]
Because
\[\int_{1}^{2{{y}_{n-1}}}{\frac{{{e}^{-{{y}_{n}}}}}{1-{{e}^{-2{{y}_{n-1}}}}}d{{y}_{n}}}=\frac{{{e}^{-1}}-{{e}^{-2{{y}_{n-1}}}}}{1-{{e}^{-2{{y}_{n-1}}}}}<{{e}^{-1}},\]
we have $I_n(a)\le e^{-1}I_{n-1}(a)$. Using this inequality repeatedly, we get
\[{{I}_{n}}(a)\le {{e}^{-n+1}}{{I}_{1}}(a)={{e}^{-n+1}}\frac{{{e}^{-1}}-{{e}^{-2a}}}{1-{{e}^{-2a}}}<{{e}^{-n}}\ \ \ \forall a>0,\,n\in\mathbb{N}.\]
Therefore, the condition (E) holds for $\psi\equiv 1$ and $r=e$. By \autoref{main}, the cascade is non-explosive.

\subsection{Birth-death branching Markov chain}
Consider a DSY cascade $\{\lambda(X_v)^{-1}T_v\}_{v\in\mathbb{T}}$ on the state space $S=\mathbb{N}$ in which the branching Markov chain $\{X_v\}_{v\in\mathbb{T}}$ has properties (A) and (B) with the transition probabilities given by
\[\mathbb{P}(X_{v*1}=j+1\,|\,X_v=j)=\mathbb{P}(X_{v*2}=j+1\,|\,X_v=j)=\beta_j,\]
\[\mathbb{P}(X_{v*1}=j-1\,|\,X_v=j)=\mathbb{P}(X_{v*2}=j-1\,|\,X_v=j)=\delta_j,\]
where $\beta_1=1$, and $\delta_j = 1-\beta_j\in (0,1)$ for $j=2,3,\dots$ Along each path $s\in\partial\mathbb{T}$, the sequence $X_{s|0}$, $X_{s|1}$, $X_{s|2}$, \ldots is the birth-death process on $S$
 with reflection at $1$ and
birth-death rates $\beta_j,\,  \delta_j$.  This is an ergodic time-reversible
Markov chain (see \cite{RB_EW2009}, Theorem 3.1(b), p. 241) with the invariant
probability
\begin{equation}
\gamma_j = \frac{\beta_2\cdots\beta_{j-1}}{\delta_2\cdots\delta_j}\gamma_1,\ \ \ j=2,3,\dots
\end{equation}
provided that
$$\gamma_1=\sum_{j=2}^\infty\frac{\beta_2\cdots\beta_{j-1}}{\delta_2\cdots\delta_j}<\infty.$$
In particular, this is the case when $\delta_2=\delta_3=\delta_4=\cdots=\delta\in(1/2,1)$. Since each state is visited infinitely often, the pathwise total waiting time  is infinite:
\[\sum_{j=0}^\infty\frac{T_{s|j}}{\lambda(X_{s|j})}=\infty \ \ \ \text{a.s.}\]
However, it will be shown below that the cascade can still be explosive.
\begin{prop}\label{expergodic}
Let $\delta_2=\delta_3=\cdots=\delta\in(0,\frac{1}{\sqrt{2}})$ and $\lambda(k)=b^k$ where $b>1$. Suppose that for each $v\in\mathbb{T}$, $X_{v*1}$ and $X_{v*2}$ are conditionally independent given $X_v$. Then the cascade is a.s.\ explosive for any initial state $X_\theta=a\in\mathbb{N}$. 
\end{prop}
\begin{proof}
One can see that condition (D) is satisfied. Put $Z=\max\{\lambda(X_1),\lambda(X_2)\}$. Then $Z^{-1}=\min\{b^{-X_1},b^{-X_2}\}$. First, consider the case $b<\delta^{-2}-1$. Put
\[c=\max\left\{\frac{1}{b},\,b\delta^2+\frac{1-\delta^2}{b}\right\}<1.\]
By \autoref{expcriterion}, it suffices to show that $\mathbb{E}_a[Z^{-1}]\le cb^{-1}$ for all $a\in\mathbb{N}$. For $a=1$, $\mathbb{E}_1[Z^{-1}]=b^{-2}\le cb^{-1}$. For $a\ge 2$,
\begin{eqnarray*}
{{\mathbb{E}}_{a}}[{{Z}^{-1}}]&=&\frac{1}{{{b}^{a-1}}}{{\mathbb{P}}_{a}}(Z={{b}^{a-1}})+\frac{1}{{{b}^{a+1}}}{{\mathbb{P}}_{a}}(Z={{b}^{a+1}})\\
&=&\frac{1}{{{b}^{a-1}}}{{\mathbb{P}}_{a}}({{X}_{1}}={{X}_{2}}=a-1)+\frac{1}{{{b}^{a+1}}}{{\mathbb{P}}_{a}}({{X}_{1}}=a+1\text{\ \ or\ \ }{{X}_{2}}=a+1)\\
&=&\frac{1}{{{b}^{a-1}}}{{\delta }^{2}}+\frac{1}{{{b}^{a+1}}}(1-{{\delta }^{2}})\le \frac{c}{{{b}^{a}}}.
\end{eqnarray*}
By \autoref{expcriterion}, the cascade is non-explosive.
Next, we consider the case $b\ge \delta^{-2}-1$. Take $1<q<\delta^{-2}-1$ and denote
\[\zeta_q=\sup_{n\ge 0}\min_{|v|=n}\sum_{j=0}^{n}q^{-X_{v|j}}T_{v|j}.\]
We proved in the first case that $\mathbb{E}_a[\zeta_q]<\infty$. Observe that $\zeta\le\zeta_q$ and thus, $\mathbb{E}_a\zeta\le\mathbb{E}_a\zeta_q<\infty$. This implies $\zeta<\infty$ a.s.
\end{proof}
\begin{prop}\label{612211}
Let $\delta_2=\delta_3=\cdots=\delta\in(\frac{2+\sqrt{3}}{4},1)$. Then for any function $\lambda:\mathbb{N}\to(0,\infty)$, a DSY cascade $\{\lambda(X_v)^{-1}T_v\}_{v\in\mathbb{T}}$ with property \textup{(C)} is non-explosive for any initial state $X_\theta=a\in\mathbb{N}$. 
\end{prop}
\begin{proof}
Any function from $\mathbb{N}$ to $(0,\infty)$ is locally bounded. By \autoref{main}, we only need to find $r>2$ and a function $\psi:\mathbb{N}\to(0,\infty)$ such that
\[\mathbb{P}_a(X_{s|1},X_{s|2},\ldots,X_{s|n}>1)\le\psi(a)r^{-n}\ \ \ \forall n,a\in\mathbb{N}.\]
As shown below, one can choose $r=\frac{1}{2\sqrt{
\delta(1-\delta)}}>2$ and $\psi(a)=\frac{a(\delta/\beta)^{(a-1)/2}}{1-2\sqrt{\beta\delta}}2\sqrt{\beta\delta}$. Fix a path $s\in\partial\mathbb{T}$ and write $X_0,X_1,X_2,\ldots$ in lieu of $X_{s|0},X_{s|1},X_{s|2},\ldots$ Denote
\[\tau=\min\{k\ge 1:\ X_k=1\}.\]
With the initial state $X_0=a\in\mathbb{N}$, $\tau$ can be viewed as the first passage time of a simple random walk on $\mathbb{Z}$ with probability of going to the left $\delta$, and probability of going to the right $\beta=1-\delta$. One has the estimates (\cite[p.\ 11]{RB_EW2009})
\begin{eqnarray*}
{{\mathbb{P}}_{a}}(\tau =N)&=&\frac{a}{N}\left( \begin{matrix}
   N  \\
   \frac{N+1-a}{2}  \\
\end{matrix} \right){{\beta }^{\frac{N+1-a}{2}}}{{\delta }^{\frac{N+a-1}{2}}}\\
&\le& a{{\left( \frac{\delta }{\beta } \right)}^{(a-1)/2}}\frac{1}{N}\left( \begin{matrix}
   N  \\
   N'  \\
\end{matrix} \right){{(\beta \delta )}^{N/2}}
\end{eqnarray*}
with $N'=[N/2]$ and with the convention that $\left( \begin{matrix}
   N  \\
   k  \\
\end{matrix} \right)=0$ if $k<0$ or $k$ is not an integer. Because $\delta>1/2$, the Markov chain is recurrent on $\mathbb{N}$. Thus, $\mathbb{P}_a(\tau=\infty)=0$. It is known that $\left( \begin{matrix}
   N  \\
   N'  \\
\end{matrix} \right)\lesssim {{2}^{N}}{{N}^{-1/2}}$
as a consequence of Sterling's formula $k!=\sqrt{2\pi k}k^ke^{-k}(1+\text{o}(k))$. Thus,
\[{{\mathbb{P}}_{a}}(\tau =N)\lesssim a{{\left( \frac{\delta }{\beta } \right)}^{(a-1)/2}}\frac{{{2}^{N}}}{{{N}^{3/2}}}{{(\beta \delta )}^{N/2}}\lesssim a{{\left( \frac{\delta }{\beta } \right)}^{a/2}}{{(2\sqrt{\beta \delta })}^{N}}.\]
We obtain
\begin{eqnarray*}
{{\mathbb{P}}_{a}}({{X}_{1}},{{X}_{2}},...,{{X}_{n}}>1)&=&{{\mathbb{P}}_{a}}(\tau >n)\\
&\lesssim& a{{\left( \frac{\delta }{\beta } \right)}^{a/2}}\sum\limits_{k=n+1}^{\infty }{{{(2\sqrt{\beta \delta })}^{k}}}\\
&\le& \psi (a){{r}^{-n}}.\end{eqnarray*}
\end{proof}
\begin{rem}
\autoref{expergodic} and \autoref{612211} do not give a conclusion about the explosion or non-explosion in the case $\delta_2=\delta_3=\cdots=\delta\in[\frac{1}{\sqrt{2}},\frac{2+\sqrt{3}}{4}]$ and $\lambda(k)=b^k$, $b>1$. However, by \autoref{01} (ii), we know that it must be a $0-1$ event.
\end{rem}

\section{DSY cascades resulting from differential equations}\label{diff_eq_examples}

Our first example serves as a precursor to the general method of associating a branching cascade structure to a quasilinear evolutionary partial differential equation. This method  goes back to McKean's treatment for the Fisher-KPP equation in the physical space \cite{mckean, MB1978} and Le Jan and Sznitman's treatment for the Navier-Stokes equations in the Fourier space \cite{lejan}.
\subsection{The $\alpha$-Riccati equation}\label{alpha-ric}
For $\alpha>0$, consider the $\alpha$-Riccati equation 
\begin{equation}\label{alpha-ricc}
u'(t)=-u(t)+u^2(\alpha t),\quad u(0)=u_0.
\end{equation} 
This equation can be viewed as a toy model for self-similar Navier-Stokes equations (see \cite{alphariccati}), It also appears from purely 
probabilistic models (see \cite{athreya, DA_PS_1988}).
After rewriting the equation in the mild formulation
\[
u(t)=u_0e^{-t}+\int_{0}^{t}e^{-s}u^2({\alpha(t-s)})\,ds,
\]
we can interpret $u$ as the expected of a stochastic functional $U(t)$ defined implicitly by
\[U(t)=u_0\,\IND{T_\theta\ge t}+U^{(1)}\left(\alpha(t-T_\theta)\right)\,U^{(2)}\left(\alpha(t-T_\theta)\right)\,\IND{T_\theta<t },\]
where $U^{(1)}$ and $U^{(2)}$ are two independent copies of $U$. Thus, the stochastic functional $U(t)$ is defined over the stochastic structure $\left\{{\alpha^{-|v|}}{T_v}\right\}_{v\in\mathbb{T}}$. In the notation of the present paper, this is a self-similar DSY cascade with $\lambda(x)=x$ and $X_v=\alpha^{|v|}$ (deterministic). The explosion problem of the $\alpha$-Riccati equation, especially in the connection with the existence and uniqueness of solutions, was studied in detail in \cite{athreya,alphariccati}. The cascade was known to be explosive if and only if $\alpha>1$. \autoref{expcriterion} provides an alternate justification.   Indeed, the ratios along each path are $R_{s|j}={X_{s|n}}/{X_{s|n-1}}=\alpha$. With $X_\theta=a>0$, $Z=\max\{X_1,X_2\}=\alpha a$. We have
\[\mathbb{E}_a[Z^{-1}]=\frac{1}{\alpha a}=\frac{c}{a}\]
with $c=1/\alpha<1$. The non-explosion of the cascade of the $\alpha$-Riccati equation for $\alpha\le 1$ can be inferred from the non-explosion of the standard Yule cascade (see \cite{complexburgers,yule}). If $\alpha<1$, \autoref{multibran} provides an alternative justification. Indeed, with $b>0$ sufficiently large, $\mathbb{E}[R_1^b]=\alpha^b<1/2$.
Note that the case $\alpha=1$ results in a classical Yule process, which is non-explosive.

In the case $\alpha\le1$, the stochastic non-explosion  was exploited in \cite{complexburgers} to prove existence and uniqueness results as well as long-time behavior of solutions to \eqref{alpha-ricc}. In the case $\alpha>1$, the stochastic explosion was used to prove a nonuniqueness result in \cite{athreya} for the case $u_0=0$ and \cite{alphariccati} developed a framework for using stochastic explosion and distribution of branches of the underlying cascade to prove both non-uniqueness and finite-time blowup of the solutions
for more general initial data.

\subsection*{DSY cascades associated with evolutionary PDEs in Fourier space}

The next examples concern the DSY cascades originating from partial differential equations.  A common feature of these equations is that they define a dissipative dynamical system in which, when formulated in the Fourier space, the linear term determines the intensity of the random waiting time and the quadratic nonlinear term leads to a random binary tree.  In all of these examples, the equations can be written in the Fourier space as
\begin{equation}
\label{generic}
\hat{u}(\xi,t) = e^{-\lambda(\xi)t} \hat{u}_0(\xi) + \int_0^t e^{-\lambda(\xi)s} \lambda(\xi) \rho(\xi) \int_{\mathbb{R}^d} B(\hat{u}(\eta,t-s), \hat{u}(\xi-\eta,t-s)) d \eta ds
\end{equation}
where $\lambda$, $\rho$ are radially symmetric positive functions, and $B$ is a bilinear map. The functions $\lambda$, $\rho$, $B$ will be determined by the specific PDE under consideration. 
A key step in the probabilistic reformulation of \eqref{generic} is to find a function $h(\xi)$ such that 
\begin{equation}\label{functionH}H(\eta|\xi) = \frac{\rho(\xi)}{ h(\xi)}h(\eta) h(\xi-\eta)\end{equation}
is a probability density function on $\mathbb{R}^d$.  Once $h$ is identified, we introduce the new unknown $\chi(\xi,t) = \hat{u}(\xi,t)/h(\xi)$, which satisfied the normalized equation
\begin{equation}
\label{genericfourier}
{\chi}(\xi,t) = e^{-\lambda(\xi)t} \chi_0(\xi) + \int_0^t e^{-\lambda(\xi)s} \lambda(\xi) \int_{\mathbb{R}^d} B(\chi(\eta,t-s), \chi(\xi-\eta,t-s)) H(\eta|\xi)\; d\eta ds.
\end{equation}

Equation \eqref{genericfourier} leads to a family of wave numbers $\{W_v\}_{v\in\mathbb{T}}$ satisfying $W_\theta=\xi$, $W_{v*1}+W_{v*2} = W_v$ for all $v\in\mathbb{T}$, and conditionally given $W_v$, $W_{v*1}$ and $W_{v*2}$ are each distributed as $H(\cdot|W_v)$. For $X_v=W_v$, one gets a DSY cascade $\{\lambda(X_v)^{-1} T_v\}_{v\in\mathbb{T}}$. In most cases, the holding times between 
branchings only depends on 
the magnitudes of the  random wave numbers which, in turn,  
have a well-behaved branching Markov structure. For example, for the Navier-Stokes equations, the choice of  $X_v=|W_v|$ turns out to be more efficient than the choice of $X_v=W_v$.

Similarly to the $\alpha$-Riccati equation (\autoref{alpha-ric}),  the solution $\chi(\xi,t)$ can be expressed as the expected value of a ``{\em solution}"  
stochastic functional $\mathcal{X}(\xi,t)$ defined over the DSY cascade by
\begin{equation}\label{gen-recursion}
\mathcal{X}(\xi ,t)=\left\{ \begin{array}{*{35}{r}}
   {{\chi }_{0}}(\xi ) & \text{if} & {{T}_{\theta }}/\lambda(\xi)\ge t  \\
   B\left({\mathcal{X}^{(1)}}({{W}_{1}},t-{{T}_{\theta }}),{\mathcal{X}^{(2)}}({{W}_{2}},t-{{T}_{\theta }})\right) & \text{if} & {{T}_{\theta }}/\lambda(\xi)< t  \\
\end{array} \right.
\end{equation}
where $\mathcal{X}^{(1)}$ and $\mathcal{X}^{(2)}$ are (conditionally) independent copies of $\mathcal{X}$.
The stochastic explosion or non-explosion of the associated DSY cascades has direct implications to the existence and uniqueness of global-in-time solutions of these equations \cite{chaos, alphariccati, smallness}.

We will illustrate the aforementioned generic scheme in greater detail for the Fisher-KPP equation (on the Fourier side) in the next subsection. Note that the stochastic structure  on the Fourier side (namely, DSY cascade) is very different  from that on the physical side (namely, branching Brownian motion) which was identified in \cite{mckean, MB1978}.  

\subsection{The Fourier-transformed KPP equation}\label{kppcascade}
The KPP equation introduced by Fisher, Kolmogorov, Petrovsky and Piskunov, also referred to as the F-KPP equation, has the general form
\[u_t=Du_{xx}+ru(1-u)~~~\forall t>0,\,x\in\mathbb{R}\]
where $D$ and $r$ are positive constant. McKean constructed a cascade model for this equation in the physical domain \cite{mckean, MB1978}. We now construct a cascade for this equation in the Fourier domain. By rescaling the time and space variables, we can assume that $D=r=1$. Then by introducing $v=1-u$, we get
\[v_t=v_{xx}+v^2-v~~~\forall t>0,\,x\in\mathbb{R}.\]
Taking Fourier transform with respect to $x$, we get
\[{{\hat{v}}_{t}}=-(1+{{\xi }^{2}})\hat{v}+\frac{1}{\sqrt{2\pi }}\hat{v}*\hat{v}.\]
In the integral form,
\[\hat{v}(\xi ,t)={{e}^{-(1+{{\xi }^{2}})t}}{{{\hat{v}}}_{0}}(\xi )+\int_{0}^{t}{\int_{-\infty }^{\infty }{{{e}^{-(1+{{\xi }^{2}})s}}\hat{v}(\eta ,t-s)\hat{v}(\xi -\eta ,t-s)d\eta ds}}.\]
We normalize $\hat v$ by $\chi(\xi,t)=\frac{\hat v(\xi,t)}{h(\xi)}$ where $h$ is a positive function to be determined. Function $\chi$ satisfies the equation
\begin{equation}\label{eq:919291}\chi (\xi ,t)={{e}^{-(1+{{\xi }^{2}})t}}{{\chi }_{0}}(\xi )+\int_{0}^{t}{\int_{-\infty }^{\infty }{(1+{{\xi }^{2}}){{e}^{-(1+{{\xi }^{2}})s}}\chi (\eta ,t-s)\chi (\xi -\eta ,t-s)H(\eta |\xi ) d\eta ds}}\end{equation}
where $H(\eta|\xi)=\frac{h(\eta)h(\xi-\eta)}{(1+\xi^2)h(\xi)}$. For $H$ to be a probability density function, $h$ must satisfy the equation \[h*h=(1+\xi^2)h.\]
The function $w(x)=\sqrt{2\pi}\check{h}(x)$ satisfies $w''=w-w^2$, which has a solution $w(x)=\frac{3}{1+\cosh x}$. It follows that
\begin{equation}\label{functionh}h(\xi)=\frac{3\xi}{\sinh(\pi\xi)}.\end{equation}
\begin{figure}
  \centering
  \includegraphics[scale=0.6]{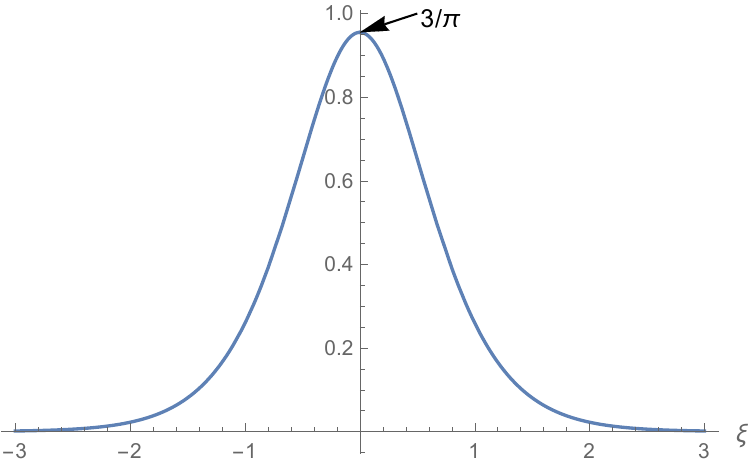}
  \caption{Graph of $h(\xi)=3\xi/\sinh(\pi\xi)$.}
  \label{graph-h}
\end{figure}
Then the representation \eqref{eq:919291} is equivalent to $\chi=\mathbb{E}_\xi[\mathcal{X}(\xi,t)]$ where $\mathcal{X}$ is a stochastic functional defined recursively by
\begin{equation}\label{recursion}
\mathcal{X}(\xi ,t)=\left\{ \begin{array}{*{35}{r}}
   {{\chi }_{0}}(\xi ) & \text{if} & {{T}_{\theta }}\ge t  \\
   {\mathcal{X}^{(1)}}({{W}_{1}},t-{{T}_{\theta }}){\mathcal{X}^{(2)}}({{W}_{2}},t-{{T}_{\theta }}) & \text{if} & {{T}_{\theta }}< t  \\
\end{array} \right.
\end{equation}
where $\mathcal{X}^{(1)}$ and $\mathcal{X}^{(2)}$ are i.i.d.\ copies of $\mathcal{X}$. This definition, when applied recursively, leads to a family of exponential mean-one clocks $\{T_v\}_{v\in\mathbb{T}}$ and a family of frequencies $\{W_v\}_{v\in\mathbb{T}}$ described as follows. For $\xi\in\mathbb{R}$ and $t>0$, we start a branching process with a particle located at $W_\theta=\xi$. The particle is time $t$ away from the horizon. The clock runs for an exponential time $\bar{T}_\theta\sim{\rm Exp}(1+\xi^2)$. If $\bar{T}_\theta>t$ then the branching process stops (no branching occurs). Otherwise, the particle dies and is replaced by two particles. A random variable $W_1\in\mathbb{R}$ is then sampled according to the p.d.f.\ $H(\eta|\xi)$. The first particle is placed at $W_1$ and the second particle is placed at $W_2=\xi-W_1$. Each particle is now $t-\bar{T}_\theta$ away from the time horizon. These particles continue to evolve by the same rule, independently of each other. Note that $T_v=\bar{T}_v/(1+W_v^2)\sim \text{Exp}(1)$. Therefore, we can associate the equation \eqref{eq:919291} with a DSY cascade $\{\lambda(X_v)^{-1}T_v\}_{v\in\mathbb{T}}$ where $X_v=W_v$, $\lambda(\xi)=1+\xi^2$ and $p(\xi,\eta)=H(\eta|\xi)$. The corresponding explosion time is
\[\zeta =\zeta (\xi )=\underset{n\to \infty }{\mathop{\lim }}\,\underset{|v|=n}{\mathop{\min }}\,\sum\limits_{j=0}^{n}{\frac{{{T}_{v|j}}}{1+W_{v|j}^{2}}}.\]
The stochastic functional $\mathcal{X}$ is a.s.\ well-defined by \eqref{recursion} for all $t>0$ if and only if $\zeta=\infty$ a.s. We will apply \autoref{main} and \autoref{criterion1} to show the non-explosion of the cascade. The same cascade was analyzed by the authors in \cite[Example 5.4]{part1_2021} and Orum in \cite[Sec.\ 7.9]{orum}. The non-explosion was proved by using a large deviation method together with a spectral radius technique in \cite{part1_2021}, and by exploiting the uniqueness of solutions to the KPP equation in \cite{orum}.

We start the proof of the non-explosion by observing that the function $h$ given by \eqref{functionh} is even and decreasing on $(0,\infty)$. According to \autoref{concave} below, $h$ is logarithmically concave on $(0,\infty)$. Thus,
\[h(\eta )h(\xi -\eta )=h(|\eta |)h(|\xi -\eta |)\le {{h}^{2}}\left( \frac{|\eta |+|\xi -\eta |}{2} \right)\le {{h}^{2}}\left( \frac{|\xi |}{2} \right)={{h}^{2}}\left( \frac{\xi }{2} \right).\]
Now fix a number $\kappa\in(1/2,1)$. Using the fact that $h$ is bounded from above by $3/\pi$, we have
\[h(\eta )h(\xi -\eta )={{[h(\eta )h(\xi -\eta )]}^{1-\kappa }}{{[h(\eta )h(\xi -\eta )]}^{\kappa }}\le {{c}_{1}}h{{(\eta )}^{1-\kappa }}{{h}^{2\kappa }}\left( \frac{\xi }{2} \right)\]
where $c_1=(3/\pi)^{1-\kappa}$. Therefore, $p(\xi,\eta)\le \phi_1(\xi)\phi_2(\eta)$ where $\phi_1(\xi)=c_1h^{2\kappa}(\xi/2)/h(\xi)$ and $\phi_2(\eta)=h(\eta)^{1-\kappa}$. Since $2\kappa>1$, $\phi_1$ decays exponentially as $\xi\to\pm\infty$. Thus, \autoref{criterion1} is satisfied with $b=1$.
\begin{lem}\label{concave}
Let $h(x)=\frac{3x}{\sinh(\pi x)}$. Then $h(x)^th(y)^{1-t}\le h(tx+(1-t)y)$ for all $x,y>0$, $t\in(0,1)$.
\end{lem}
\begin{proof}
Put $k(x)=\log h(x)$. By basic differentiations,
$k''(x)=-\frac{1}{x^2}+\frac{\pi^2}{\sinh^2(\pi x)}$.
Using the fact that $\sinh(\pi x)>\pi x$ for all $x>0$, we have
\[k''(x)<-\frac{1}{x^2}+\frac{\pi^2}{(\pi x)^2}=0.\]
Therefore, $k$ is concave on $(0,\infty)$.
\end{proof}

\subsection{The complex Burgers equation}
Consider the one-dimensional Burgers equation \[u_t=u_{xx}-uu_x,\ \ \ u(x,0)=u_0(x),\] where $u=u(x,t)$ is a complex-valued function. In the analysis of self-similar solutions in the Fourier domain \cite{complexburgers}, the equation is associated with a self-similar DSY cascade 
$\left\{ \lambda({X_v})^{-1}{{{T}_{v}}} \right\}_{v\in \mathbb{T}}$
in which $\lambda(x)=x^2$ and the ratios $R_{s|n}={X_{s|n}}/{X_{s|n-1}}$ are uniformly distributed on $(0,1)$. Since $\mathbb{E}[R_{1}^2]=1/3$, the cascade is non-explosive according to \autoref{multibran}.

\subsection{Bessel cascade of the three-dimensional NSE}\label{besselcascade}
The $d$-dimensional incompressible Navier-Stokes equations are given by

  \begin{equation}\tag{NSE}\label{NSE}
  \left\{ {\begin{array}{*{20}{rcl}}
   {{\partial _t}u  + u\cdot \nabla u  = \Delta u-\nabla p}&~~{\rm in}&\mathbb{R}^d\times(0,\infty),  \\
   {{\rm div}\, u = 0}&~~{\rm in}&\mathbb{R}^d\times(0,\infty),  \\
   {u(\cdot,0) = {u_0}}&~~{\rm in}&\mathbb{R}^d.  \\
\end{array}} \right.
\end{equation}
In the general formulation \eqref{genericfourier} for the Navier-Stokes equations, $\rho(\xi)=|\xi|^{-1}$ and the branching distribution of wave numbers is given by
\begin{equation}\label{branchingdist}H(\eta |\xi )=\frac{h(\eta )h(\xi -\eta )}{|\xi |h(\xi )}.\end{equation}
The function $h:\mathbb{R}^d\backslash\{0\}\to(0,\infty)$ satisfies the equation $h*h=|\xi|h$ and is called a \emph{standard majorizing kernel} \cite{rabi}. In three dimensions, the Bessel cascade is the DSY cascade corresponds to the choice of $h(\xi)=\frac{1}{2\pi}\frac{e^{-|\xi|}}{|\xi|}$ (see e.g. \cite{rabi, chaos, lejan}, \cite[Prop.\ 3.8]{orum}). In the analysis of the explosion problem, the choice $X_v=|W_v|$ turns out to be more efficient than the choice $X_v=W_v$. The corresponding function $\lambda$ is $\lambda(x)=x^2$. With $y=|\eta|$ and $x=|\xi|$, the branching distribution of $X_1=|W_1|$ given $X_\theta=x$ is (see \cite{chaos}):
 \[p\left( x,y \right) =  \left\{ \begin{array}{*{35}{l}}
   \frac{{{e}^{2x}}-1}{x}{{e}^{-2y}} & \text{if} & 0<x\le y,  \\[6pt]
   \frac{1-{{e}^{-2y}}}{x} & \text{if} & 0<y<x.  \\
\end{array}\right . \]
Along each path $s\in\partial\mathbb{T}$, the sequence $\{X_{s|j}\}_{j\ge 0}$ is a time-reversible Markov chain with the invariant probability density
\[\gamma(x)=4xe^{-2x}.\]
The non-explosion of the Bessel cascade was proved in \cite[Example 5.2]{part1_2021} via a large deviation method. Here we present an alternative proof using \autoref{main} and \autoref{criterion}. The conditions (i)--(v) in \autoref{criterion} are satisfied for the choice of
\[\psi_1(x)=\frac{e^{2x}-1}{4x},\ \psi_2(x)=e^{-2x},\ c_1=c_2=2,\,\alpha=5.\]
Therefore, the DSY cascade $\left\{X_v^{-2}T_v\right\}_{{v\in\mathbb{T}}}$ is non-explosive.

\section{Cascade for the Navier-Stokes equations with scale-invariant kernel}
\label{NSEExplosion}
This section explores the explosion/non-explosion problem of a self-similar DSY cascade naturally associated with the $d$-dimensional incompressible Navier-Stokes equations with $d\ge 3$.
Based on the general scheme \eqref{genericfourier}, the present cascade corresponding to $\rho(\xi)=|\xi|^{-1}$ and the choice of standard majorizing kernel $h(\xi)=C_d|\xi|^{1-d}$, called the \emph{scale-invariant kernel}. 
The probability density function $H(\cdot|\xi)$ is
\begin{equation}\label{ssH_d}
H(\eta|\xi)=C_d\frac{|\xi|^{d-2}}{|\eta|^{d-1}|\xi-\eta|^{d-1}}.
\end{equation} 
In contrast to the Bessel kernel discussed above, this family of kernels scales according to the natural scaling of (NSE), namely 
\[u(x,t)\to\kappa u(\kappa x,\kappa^2 t),~~p(x,t)\to\kappa^2 p(\kappa x,\kappa^2 t),~~u_{0}(x)\to\kappa u_0(\kappa x),\ \ \kappa\in\mathbb{R}.\]
Choose the branching Markov chain $X_v=|W_v|$. The corresponding intensity function is $\lambda(x)=x^2$. In the case $d=3$, the ratios $R_{s|n}={X_{s|n}}/{X_{s|n-1}}$ have the dilogarithmic distribution with density
 \begin{equation}\label{dilogfunc}f_3(r)=\frac{2}{\pi^2}\frac{1}{r}\ln\frac{r+1}{|r-1|}\end{equation}
 and the corresponding self-similar DSY cascade is referred to as the {\em dilogarithmic cascade} \cite{chaos}. It was shown in \cite{chaos} that, for $d=3$, the event of non-explosion $[\zeta=\infty]$ is a $0-1$ event and that a.s.\ explosion does not occur along any deterministic path $s\in\partial\mathbb{T}$, i.e.\ $\sum_{j=0}^\infty X_{s|j}^{-2}T_{s|j}=\infty$ (see \cite[Cor.\ 5.1]{chaos}). \autoref{01} (i) implies that the non-explosion is a $0-1$ event for any dimension $d\ge 3$. We will show that non-explosion occurs with probability one for $d\ge 12$ (\autoref{nonexp11}), and with probability zero for $d=3$ (\autoref{exp3}).

 \begin{rem}
 In the construction of solutions to the 3-dimensional incompressible Navier-Stokes equations, Le Jan and Sznitman circumvented the problem of stochastic explosion by introducing a clever, albeit adhoc, 
critical coin tossing device. This technique assures the a.s.\ termination of the branching after finitely many steps.
Subcritical coin tosses could also be implemented for the 
same effect (see \cite{rabi}).   The elimination of the coin-toss device naturally leads to 
an intrinsic explosion problem for the branching cascade itself.  This problem depends on the Markov transition probabilities of the wave numbers $H(\cdot|\xi)$, which is determined by the choice of $h$. In the construction of global or local-in-time solutions, Bhattacharya et al \cite{rabi} relaxed the condition on $h$ to only $h*h\le C|\xi|^\theta h$ for some constants $\theta\le 1$ and $C>0$ (called the \emph{majorizing kernel}). Majorizing kernels determine the space of initial data suitable for the construction of solutions. 
While our focus here is on zero forcing, nonzero forcing can also be accommodated
in this framework (see \cite{chaos}).
\end{rem}
 
 \medskip
 
We will derive the branching distribution of $X_1=|W_1|$ given $X_\theta=|\xi|$ as follows. Write $x=|\xi|$, $y=|\eta|$ and $z=|\xi-\eta|$ and let $\phi$ be the angle between $\xi$ and $\eta$. Then $z=\sqrt{x^2-2xy\cos\phi+y^2}$. Put $k(x)=h(\xi)=C_dx^{1-d}$. Choose the spherical coordinates in $\mathbb{R}^d$ by $\eta=(y\cos\phi,\,y(\sin\phi)w)$ where $w\in \mathbb{S}^{d-2}$, the unit sphere in $\mathbb{R}^{d-1}$. Then $d\eta=y^{d-1}\sin^{d-2}\phi\, dy d\phi dw$. 
Therefore, 
\begin{equation}\label{W-dist2}
H(\eta|\xi)\,d\eta = {{C}_{d}}{\frac{x^{d-2}{{\sin }^{d-2}}\phi }{{{\left( {{x}^{2}}-2xy\cos \phi +{{y}^{2}} \right)}^{\frac{d-1}{2}}}}\,dy d\phi d w }.
\end{equation}
For any smooth function $g(y)$,
\begin{eqnarray*}
{{\mathbb{E}}_{x}}g(|{{W}_{1}}|)&=&\int_{{{\mathbb{R}}^{d}}}{g(|\eta |)H(\eta|\xi)\,d\eta }\\
&=&\int_{0}^{\infty }{\int_{0}^{\pi }{\int_{{\mathbb{S}^{d-2}}}^{{}}{g(y)\, {C}_{d}}{\frac{x^{d-2}{{\sin }^{d-2}}\phi }{{{ \left( {{x}^{2}}-2xy\cos \phi +{{y}^{2}} \right)}^{\frac{d-1}{2}}   }}\,  dwd\phi dy}}}\\
&=&|{\mathbb{S}^{d-2}}|{{C}_{d}}{{x}^{d-1}}\int_{0}^{\infty }{\int_{0}^{\pi }{g(y)\frac{{{\sin }^{d-2}}\phi }{{  \left( x^{2}-2xy\cos \phi +y^2 \right)^{\frac{d-1}{2}} }}d\phi dy}}.
\end{eqnarray*}
Thus, the distribution of $|W_1|$ given $|\xi|$ has a density 
\begin{equation}\label{X-dist1}
p\left( x,y \right)={{c}_{d}}{{x}^{d-1}}\int_{0}^{\pi }{\frac{{{\sin }^{d-2}}\phi }{{{\left( {{x}^{2}}-2xy\cos \phi +{{y}^{2}} \right)}^{\frac{d-1}{2}}}}d\phi },
\end{equation}
where 
\begin{equation}\label{c_d}
{{c}_{d}}=|{\mathbb{S}^{d-2}}|{{C}_{d}}=\frac{\Gamma \left( \frac{d-1}{2} \right)}{\Gamma \left( \frac{d-2}{2} \right)}\frac{2}{{{\pi }^{3/2}}}.
\end{equation}
The ratios $R_{s|n}={X_{s|n}}/{X_{s|n-1}}$ have a density
\begin{equation}\label{density}f_d(r)=p(1,r)={{c}_{d}}\int_{0}^{\pi }{\frac{{{\sin }^{d-2}}\phi }{{{\left( {{1}}-2r\cos \phi +{{r}^{2}} \right)}^{\frac{d-1}{2}}}}d\phi }\end{equation}
which is independent of $s\in \partial\mathbb{T}$ and $n\in\mathbb{N}$.  Therefore, for any $d\ge 3$, $\{X_v^{-2}T_v\}_{v\in\mathbb{T}}$ is a self-similar DSY cascade with the ratio distribution given by \eqref{density}. 
\begin{rem}
One can check with basic calculations that for $d=3$, the distribution of $R_1$ is 
symmetric with respect to the group identity $r=1$ on $(0,\infty)$, i.e.\ the median of $R_1$ is equal to 1. If $d\ge 4$, one can show $\int_0^1f_d(r)dr>\int_1^\infty f_d(r)dr$ by using the change of variable $r\mapsto 1/r$. Thus, the median of $R_1$ is less than 1, which better facilitates the non-explosion of the cascade. At this heuristic level, the non-explosion is more likely to happen for $d\ge 4$ than for $d=3$. One may suspect that for very large $d$, the self-similar cascade of the Navier-Stokes equations is non-explosive. The problem will be further discussed in the next section.
\end{rem}

\medskip

Besides the spherical coordinates, one can also express $H(\eta|\xi)d\eta$ in the coordinates $(\phi_1=\phi,\phi_2,w)$ where $\phi_2$ is the angle between $\xi$ and $\xi-\eta$. Using the sine rule in the triangle made by $\xi$, $\eta$, $\xi-\eta$, we get 
\[y=x\frac{\sin \phi_2}{\sin(\pi-\phi_1-\phi_2)}=x\frac{\sin\phi_2}{\sin(\phi_1+\phi_2)}.\]
Then changing variables in (\ref{W-dist2}) from $(y,\phi,w)$ to $(\phi_1,\phi_2,w)$, noting the Jacobian $\frac{\partial(y,\phi,w)}{\partial(\phi_1,\phi_2,w)}=\frac{-x\sin\phi_1}{\sin^2(\phi_1+\phi_2)}$, we obtain
 \begin{equation}\label{W-dist3}
H(\eta|\xi)\,d\eta = {{C}_{d}}\,{\sin }^{d-3}(\phi_1+\phi_2) \,d\phi_1 d \phi_2 dw ,
\end{equation}
where $(\phi_1,\phi_2)\in \Delta:=\{(\phi_1,\phi_2):\ \phi_1,\phi_2\ge 0,\ \phi_1+\phi_2\le\pi\}$. This is the triangle formed by 
$\xi/|\xi|,$ $\eta/|\xi|$, and $(\xi-\eta)/|\xi|$. As $\eta$ is randomized according to $H(\cdot|\xi)$, $\phi_1$ and $\phi_2$ are values of the corresponding random variables $\Phi_1$ and $\Phi_2$. The joint distribution density of the random angles $(\Phi_1,\Phi_2)$ conditionally given $e_{\xi}=\xi/|\xi|$ is
\begin{equation}\label{angle-distr}
S_d(\phi_1,\phi_2)=c_d {\sin }^{d-3}(\phi_1+\phi_2), 
\end{equation}
where $c_d$ is given by (\ref{c_d}).
It is easy to see that $(\Phi_1,\Phi_2)$ is uniformly distributed in the triangle $\Delta$ if and only if  $d=3$. 

\subsection{Non-explosion of the self-similar Navier-Stokes cascades in high dimensions}
As $d\to\infty$, $S_d$ given by \eqref{angle-distr} approaches the uniform distribution on the line segment $\{\phi_1,\phi_2\ge 0,\, \phi_1+\phi_2=\pi/2\}$. Geometrically, the triad $\xi$, $W_1$, $W_2$ tend to form a right triangle with sides with hypotenuse $\xi$. Consequently, $R_{1}$ and $R_{2}$ become bounded by 1 in the limit $d\to\infty$. The explosion time of the self-similar cascade of the Navier-Stokes becomes bounded by that of the classical Yule cascade, or equivalently the $\alpha$-Riccati cascade with $\alpha=1$, which is non-explosive. 
Hence, it is quite conceivable that the self-similar DSY cascade of the Navier-Stokes equations is non-explosive if $d$ is sufficiently large. The following result affirms that this is the case. 
\begin{prop}\label{nonexp11}
For $d\ge 12$, the self-similar cascade of the Navier-Stokes equations in $\mathbb{R}^d$ is a.s.\ non-explosive for any initial state $\xi\in\mathbb{R}^d\backslash\{0\}$.
\end{prop}
\begin{proof}
As shown above, $R_1$ has a density function $f_d(r)$ given by \eqref{density}. By \autoref{multibran}, it is sufficient to show that $\alpha_d:=\mathbb{E}[R_1^{(d-3)/2}]<1/2$. With the change of variable $\phi\mapsto s=\sqrt{1-2r\cos\phi+r^2}$, one gets
\[{{r}^{\frac{d-3}{2}}}{{f}_{d}}(r)=\frac{{{c}_{d}}}{r{{(4r)}^{\frac{d-3}{2}}}}\int_{|1-r|}^{1+r}{\frac{{{({{(1+r)}^{2}}-{{s}^{2}})}^{\frac{d-3}{2}}}{{({{s}^{2}}-{{(1-r)}^{2}})}^{\frac{d-3}{2}}}}{{{s}^{d-2}}}ds}.\]
Put $\bar{f}_{d}(r)=r^{(d-3)/2}f_d(r)$. We show $\bar{f}_d(r)\le \bar{f}_{d-2}(r)$ for all $d\ge 5$, $r>0$. By the fact that $\left(-\frac{1}{(d-3)s^{d-3}}\right)'=\frac{1}{s^{d-2}}$ and by integration by parts, one obtains
\begin{eqnarray*}\frac{r{{(4r)}^{\frac{d-3}{2}}}}{{{c}_{d}}}{\bar{f}_{d}}(r)&=&\int_{|1-r|}^{1+r}{\frac{{{({{(1+r)}^{2}}-{{s}^{2}})}^{\frac{d-3}{2}}}{{({{s}^{2}}-{{(1-r)}^{2}})}^{\frac{d-3}{2}}}}{{{s}^{d-2}}}ds}\\
&=&\int_{|1-r|}^{1+r}{\frac{{{({{(1+r)}^{2}}-{{s}^{2}})}^{\frac{d-5}{2}}}{{({{s}^{2}}-{{(1-r)}^{2}})}^{\frac{d-5}{2}}}}{{{s}^{d-4}}}(\underbrace{2+2{{r}^{2}}-2{{s}^{2}}}_{\le 4r})ds}\\
&\le& \frac{r{{(4r)}^{\frac{d-3}{2}}}}{{{c}_{d}}}{\bar{f}_{d-2}}(r).\end{eqnarray*}
Thus, $\alpha_d\le \alpha_{d-2}$ for all $d\ge 5$. The next step is to evaluate $\alpha_d$ for some values of $d$. 
By the sine rule in the triangle with edges $1,R_1,R_2$ (\autoref{Z_calc1}), ${{R}_{1}}=\frac{\sin {{\Phi }_{2}}}{\sin ({{\Phi }_{1}}+{{\Phi }_{2}})}$. Thus,
\begin{eqnarray*}
{{\alpha }_{d}}&=&\iint\limits_{\Delta }{{{\left( \frac{\sin {{\phi }_{2}}}{\sin ({{\phi }_{1}}+{{\phi }_{2}})} \right)}^{\frac{d-3}{2}}}{{S}_{d}}({{\phi }_{1}},{{\phi }_{2}})d{{\phi }_{1}}d{{\phi }_{2}}}={{c}_{d}}\iint\limits_{\Delta }{{{\left( \sin {{\phi }_{2}}\sin ({{\phi }_{1}}+{{\phi }_{2}}) \right)}^{\frac{d-3}{2}}}d{{\phi }_{1}}d{{\phi }_{2}}}.
\end{eqnarray*}
We obtain the numerical approximations
\[\alpha_{10}\approx 0.5427,\ \ \ \alpha_{11}\approx 0.5143,\ \ \ \alpha_{12}\approx 0.4898,\ \ \ \alpha_{13}\approx 0.4684.\]
This implies that $\alpha_d<1/2$ for all $d\ge 12$.
\end{proof}
\begin{rem}
Although an explicit formula for $\alpha_d$ is not needed for the above proof, one can obtain it with the aid of Mathematica:
\[{{\alpha }_{d}}=\frac{{{2}^{\frac{d-3}{2}}}\Gamma {{\left( \frac{d-1}{4} \right)}^{3}}}{\pi \Gamma \left( \frac{d-2}{2} \right)\Gamma \left( \frac{d+1}{4} \right)}.\]
\end{rem}
\begin{rem}
As mentioned before, \eqref{angle-distr} implies that, as $d\to\infty$, the random vector $(\Phi_1,\Phi_2)$ converges in distribution to a vector $(\Phi^\infty_1,\Phi_2^\infty)$ distributed uniformly on the line segment $\{\phi_1,\phi_2\ge 0,\, \phi_1+\phi_2=\pi/2\}$. The corresponding random variables $R_1$ and $R_2$ would then represent the sides of the right triangle with hypothenuse 1 and adjacent angles $\Phi^\infty_1$, $\Phi_2^\infty$. This configuration gives raise to a self-similar DSY cascade that can be viewed as the limit case of the cascades associated with \eqref{ssH_d}, corresponding to $d=\infty$. The mean-field model of this cascade for $d=\infty$ yields $R_1=R_2\equiv 1/\sqrt{2}$, which is exactly the DSY cascade of the $\alpha$-Riccati equation with $\alpha=1/2$. The special case $\alpha=1/2$ has several critical properties in the context of $\alpha$-Riccati cascades. First, this continuous-time Markov process is a Poisson process with unit intensity. Second, the infinitesimal generator for the  Markov process defined by the 
set of vertices alive at time $t$ is a bounded operator if
and only if $\alpha\le 1/2$ (see \cite{alphariccati, yule}).   
\end{rem}

\subsection{Explosion of the self-similar Navier-Stokes cascade in dimension $d=3$}
 
 In the case $d=3$, the self-similar DSY cascade (or dilogarithmic cascade) has several special properties. First, the ratios $R_{s|n}=X_{s|n}/X_{s|n-1}$ have a dilogarithmic distribution with density given by \eqref{dilogfunc}.
Second, conditionally given $W_v$,  the angle $\Phi_1$ between $W_{v*1}$ and  $W_{v}$ and the angle $\Phi_2$ between $W_{v*2}$ and  $W_{v}$ have a joint uniform distribution 
 in 
\[\Delta=\{(\phi_1,\phi_2):\ \phi_1,\phi_2\ge 0,\ \phi_1+\phi_2\le\pi\}.\]
We will use \autoref{expcriterion} to show that the dilogarithmic cascade is a.s.\ explosive.
Consider the random variable $R_{\max}=\max\{R_1,\,R_2\}$. Recall that  
 $R_1=|W_1|/|\xi|$ and $R_2={|W_2|}/{|\xi|}$.
By the triangular inequality, $R_{\max}\in[1/2,\infty)$. Our first step is to compute the distribution of $R_{\max}$.
\begin{lem}\label{Z-density}
The random variable $R_{\max}=\max\{R_1,\,R_2\}$ has a probability density
\[
g(r)=\frac{4}{\pi^2}\frac{1}{r}\ln\frac{r}{|r-1|}\IND{r\ge1/2}.
\]
\end{lem}
\begin{proof}
For any $r>0$, denote
\[G(r):=\mathbb{P}(R_{\max}\le r)=2\mathbb{P}(R_1\le R_2\le r)\,.\]
The event $[R_1\le R_2\le r]$ can be written as $[0\le\Phi_2\le\phi_2^*(r),\ \Phi_2\le\Phi_1\le\phi_1^*(r,\Phi_2)]$ (see \autoref{Z_calc1}). 
\begin{figure}[h]
\begin{center}
\includegraphics[scale=.8]{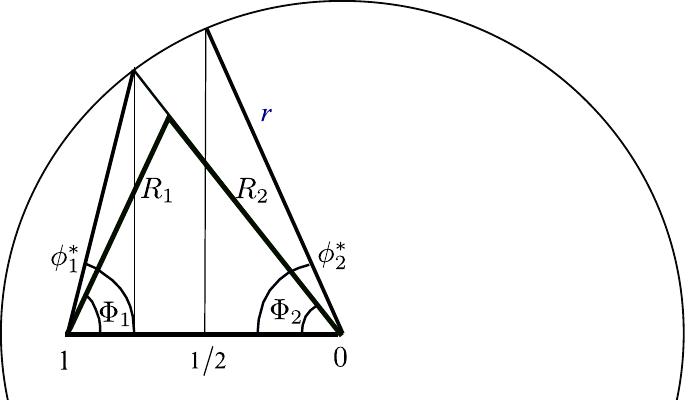} 
\end{center}
\vspace{-.2in}
\caption{An illustration of the event $[R_1\le R_2\le r]$.}\label{Z_calc1}
\end{figure}
We observe that
\[\cos\phi_2^*=\frac{1}{2r},\ \ \ \tan\phi_1^*=\frac{r\sin\phi_2}{1-r\cos{\phi_2}}\,.\]
Thus, 
\[G(r)=\frac{2}{|\Delta|}\int_0^{\phi_2^*}\int_{\phi_2}^{\phi_1^*}{d} \phi_1 {d} \phi_2=\frac{4}{\pi^2} \int_0^{\phi_2^*} \left(\phi_1^*-\phi_2\right)\,{d} \phi_2\,.
\]
By the chain rule of differentiation,
\begin{equation}\label{eq:722214}
\frac{dG}{dr}=\frac{4}{\pi^2}\int_0^{\phi_2^*(r)}\frac{\partial \phi_1^*}{\partial r} (r,\phi_2)\, d\phi_2
+\frac{4}{\pi^2}\left(\phi_1^*(r,\phi_2^*)-\phi_2^*\right)\,.
\end{equation}
Since $\phi_1^*(r,\phi_2^*)=\phi_2^*$, the second term on the right hand side is zero. To compute ${\partial \phi_1^*}/{\partial r}$, we note that on one hand
\begin{equation}\label{eq:722211}
\frac{\partial }{\partial r} \tan\left(\phi_1^*(r,\phi_2) \right)=\frac{\partial }{\partial r}\frac{r\sin\phi_2}{1-r\cos{\phi_2}}
=\frac{\sin\phi_2}{(1-r\cos{\phi_2)^2}}\,,
\end{equation}
 and on the other hand (the chain rule),
\begin{equation}\label{eq:722212}
\frac{\partial }{\partial r} \tan\left(\phi_1^*(r,\phi_2)\right) =\frac{1}{\cos^2\left(\phi_1^*(r,\phi_2)\right)}\frac{\partial \phi_1^*}{\partial r}(r,\phi_2)\,.
\end{equation}
Observe that (\autoref{Z_calc1})
\begin{equation}\label{eq:722213}
\cos^2\left(\phi_1^*(r,\phi_2)\right)=\frac{(1-r\cos\phi_2)^2}{(1-r\cos^2\phi_2)^2+(r\sin\phi_2)^2}=\frac{(1-r\cos\phi_2)^2}{1+r^2-2r\cos\phi_2}\,.
\end{equation}
By \eqref{eq:722211}, \eqref{eq:722212}, \eqref{eq:722213}, we obtain
\[
\frac{\partial \phi_1^*}{\partial r} (r,\phi_2)=\frac{\sin \phi_2}{1+r^2-2r\cos\phi_2}\,.
\] 
Now substitute this result into \eqref{eq:722214}:
\[
\frac{dG}{d r} = \frac{4}{\pi^2}\int_0^{\phi_2^*}\frac{\sin \phi_2}{1+r^2-2r\cos\phi_2}\,d\phi_2
=\frac{4}{\pi^2}\frac{1}{2r}\int_1^{2r}\frac{d u}{1+r^2-u}=\frac{4}{\pi^2}\frac{1}{r}\ln\frac{r}{|r-1|}\,.
\]
\end{proof}
\begin{rem} 
If we change variable $\tilde{R}=2R_{\max}-1$, then the distribution density of $\tilde{R}$ is given by
\[
\tilde{g}(r)=\frac{2r}{1+r}\frac{2}{\pi^2}\frac{1}{r}\ln\left|\frac{r+1}{r-1}\right|,\quad  r\ge 0,
\]
which is a tilted dilogarithmic distribution. 
Recall that $R_1$ and $R_2$ individually have the dilogarithmic distribution.
\end{rem}
\begin{prop}\label{exp3}
The self-similar cascade of the Navier-Stokes equations in $\mathbb{R}^3$ is a.s.\ explosive for any initial state $\xi\in\mathbb{R}^3\backslash\{0\}$.
\end{prop}
\begin{proof}
By \autoref{Z-density},
\begin{equation}\label{EXP_1/Z^2}
\mathbb{E}\left[R_{\max}^{-2}\right]=\int_{1/2}^{\infty} r^{-2} g(r)\, d r =\frac{8}{\pi^2}<1.
\end{equation}
We can now apply \autoref{expcriterion} with $X_\theta=|\xi|$, $\lambda(x)=x^2$, $Z=|\xi|^2R_{\max}^2$, and $\kappa=8/\pi^2$.
\end{proof}
Our explosion criterion (\autoref{expcriterion}) is not satisfied for dimensions $d\ge 4$. The following proposition provides another way to obtain the expectation in \eqref{EXP_1/Z^2} for $d=3$ and also shows that $d=4$ is the critical case for our method.
\begin{prop}
For $d\ge 3$, we have
\[{{\kappa }_{d}}:=\mathbb{E}[R_{\max }^{-2}]=\frac{4}{\pi }\frac{\Gamma {{\left( \frac{d-1}{2} \right)}^{2}}}{\Gamma \left( \frac{d-2}{2} \right)\Gamma \left( \frac{d}{2} \right)}.\]
The sequence $\{\kappa_d\}_{d\ge 3}$ is strictly increasing with $\kappa_3=8/\pi^2$, $\kappa_4=1$ and $\lim_{d\to\infty}\kappa_d=4/\pi$.
\end{prop}
\begin{proof}
By the sine rule in the triangle with edges $1,R_1,R_2$ (\autoref{Z_calc1}), ${{R}_{1}}=\frac{\sin {{\Phi }_{2}}}{\sin ({{\Phi }_{1}}+{{\Phi }_{2}})}$ and ${{R}_{2}}=\frac{\sin {{\Phi }_{1}}}{\sin ({{\Phi }_{1}}+{{\Phi }_{2}})}$. 
On the triangle $\Delta_1=\{(\phi_1,\phi_2)\in\Delta:\phi_1\ge\phi_2\}$, we have $R_{\max}=R_2$. Thus,
\[{{\kappa }_{d}}=2\iint\limits_{{{\Delta }_{1}}}{r_{2}^{-2}{{S}_{d}}({{\phi }_{1}},{{\phi }_{2}})d{{\phi }_{2}}d{{\phi }_{1}}}=2{{c}_{d}}\int_{0}^{\pi }{\int_{0}^{{{\phi }_{1}}}{\frac{{{\sin }^{d-1}}({{\phi }_{1}}+{{\phi }_{2}})}{{{\sin }^{2}}{{\phi }_{1}}}d{{\phi }_{2}}d{{\phi }_{1}}}}=2{{c}_{d}}\int_{0}^{\pi /2}{\frac{K({{\phi }_{1}})}{{{\sin }^{2}}{{\phi }_{1}}}d{{\phi }_{1}}}\]
where $S_d$ is 
given by \eqref{angle-distr} and $K(\phi )=\int_{0}^{\phi }{({{\sin }^{d-1}}(\phi +{{\phi }_{2}})+{{\sin }^{d-1}}(\phi -{{\phi }_{2}}))d{{\phi }_{2}}}$. One can express $K$ in terms of the hypergeometric function as follows.
\[K(\phi )=\left\{ \begin{array}{*{35}{r}}
   \widetilde{K}(\phi ) & \text{if} & 0\le \phi \le \frac{\pi }{4},  \\
   \frac{\sqrt{\pi }\Gamma \left( \frac{d}{2} \right)}{\Gamma \left( \frac{d+1}{2} \right)}-\widetilde{K}(\phi ) & \text{if} & \frac{\pi }{4}\le \phi \le \frac{\pi }{2}  \\
\end{array} \right.\]
where $\widetilde{K}(\phi )=\frac{1}{d}\,{{\,}_{1}}{{F}_{2}}\left( \frac{1}{2},\frac{d}{2},\frac{d+2}{2};{{\sin }^{2}}(2\phi ) \right){{\sin }^{d}}(2\phi )$. One can now compute $\kappa_d$ either with the aid of Mathematica or directly by induction on even and odd $d$:
\[{{\kappa }_{d}}=2{{c}_{d}}\int_{0}^{\pi /4}{\left( \frac{K(\phi )}{{{\sin }^{2}}\phi }+\frac{K(\pi /2-\phi )}{{{\cos }^{2}}\phi } \right)d\phi =}\frac{4}{\pi }\frac{\Gamma {{\left( \frac{d-1}{2} \right)}^{2}}}{\Gamma \left( \frac{d-2}{2} \right)\Gamma \left( \frac{d}{2} \right)}.\]
Using the asymptotic approximation $\frac{\Gamma(x+1/2)}{\Gamma(x)}\sim\sqrt{x}$ as $x\to\infty$, one gets $\kappa_d\to \frac{4}{\pi}$ as $d\to\infty$. To show the monotonicity, we use Kershaw's inequality $\frac{\Gamma(x+1)}{\Gamma(x+1/2)}>(x+\frac{1}{4})^{1/2}$ for all $x>0$ (see \cite{kershaw}).
\[\frac{{{\kappa }_{d+1}}}{{{\kappa }_{d}}}=\frac{4}{(d-1)(d-2)}{{\left( \frac{\Gamma \left( \frac{d}{2} \right)}{\Gamma \left( \frac{d-1}{2} \right)} \right)}^{4}}>\frac{4}{(d-1)(d-2)}{{\left( \frac{d}{2}-\frac{3}{4} \right)}^{2}}>1.\]
\end{proof}

\begin{rem}
As mentioned in the introduction, there is a strong connection between the stochastic explosion of the DSY cascade underlying a certain evolution PDE and the well-posedness  problem for the PDE itself. In particular, if we simplify the mild-type formulation \eqref{generic} corresponding to the Navier-Stokes equations in Fourier space by replacing $\hat{u}$ with a scalar function and replacing the sophisticated product-like structure in $B(\hat{u},\hat{u})$ (coming from the Fourier transform of $(u\cdot\nabla) u$) with the a simple product of scalars, we will obtain a 
non-linear scalar PDE with {\em exactly the same DSY cascade structure} as the Navier-Stokes equations \eqref{NSE}. This scalar PDE was first considered by Montgomery-Smith \cite{smith} as a model for a finite-time blowup.
As shown in \cite{smallness}, the simplified product structure of $B$ 
allows to recover the blowup results from \cite{smith} {\em directly} from the DSY cascade. Also, the explosion of the self-similar DSY cascade (\autoref{exp3}) yields a nonuniqueness result for the initial value problem of the Montgomery-Smith equation \cite{smallness}. 

In the case of 3-dimensional Navier-Stokes equations, the problem of existence and uniqueness of global solutions naturally involves possibility of cancellations in 
\eqref{gen-recursion} due to the geometric structure of the bilinear product $B$. 
While the present paper focussed on the time evolution of the DSY cascade process itself, 
it remains to be determined whether the additional geometric structure emanating from the nonlinear term $B$ corresponding to 
 \eqref{NSE} provides sufficient cancellations in some smooth initial data to negate the impact of explosion of the self-similar DSY cascade.
\end{rem}

\bibliography{References}
\end{document}